\documentclass{amsart}
\usepackage{amsmath}
\usepackage{amssymb,bbding}
\usepackage{amsthm,fancyvrb,rotating}
\usepackage[all]{xy}

\usepackage{graphicx}
\date{}
\def\ori{\psi}
\def\proj{\qopname\relax o{pr}}
\def\id{\qopname\relax o{id}}
\def\sing{\qopname\relax o{sing}}
\def\phigeom{\Phi_{\mathrm{g}}}
\def\core{\qopname\relax o{core}}
\def\exc{\qopname\relax o{ex}}

\theoremstyle{theorem}
\newtheorem{theorem}{Theorem}
\newtheorem{lemma}{Lemma}
\theoremstyle{remark}
\newtheorem{remark}{Remark}

\theoremstyle{definition}
\newtheorem{definition}{Definition}
\newtheorem{proposition}{Proposition}
\author {Ivan Dynnikov and Alexandra Skripchenko}
\address{\noindent V.A. Steklov Mathematical Institute of Russian Academy of Science, 8 Gubkina Str., Moscow 119991, Russia}
\email{dynnikov@mech.math.msu.su}
\address{CNRS UMR 7586, Institut de Math\'ematiques de Jussieu---Paris Rive Gauche, Batiment Sophie Germaine, Case 7021, 75205 Paris Cedex 13, France, and}
\address{Laboratory of Geometric Methods in Mathematical Physics, Lomonosov Moscow State University, Moscow 119991, Russia}
\email{sashaskrip@gmail.com}

\title[{On typical leaves of a  measured foliated 2-complex of thin type}]{On typical leaves of a  measured foliated 2-complex\\ of thin type}
\dedicatory{On the occasion of S.P.Novikov's 75th birthdate}

\begin{document}
\maketitle

\begin{abstract}
It is known that all but finitely many leaves of a measured foliated
2-complex of thin type are quasi-isometric to an infinite tree with at most two topological
ends. We show that if the foliation is cooriented, and the
associated $\mathbb R$-tree is self-similar, then a typical
leaf has exactly one topological end. We also construct the first example of a foliated 2-complex
of thin type whose typical leaf has exactly two topological ends. `Typical' means that
the property holds with probability one in a natural sense.
\end{abstract}

\section{Introduction}
This work grew out from an attempt to understand the behavior of plane sections of a 3-periodic surface in $\mathbb R^3$
in the so-called chaotic case \cite{dyn08}. The general problem on the structure of such sections was
brought to mathematics by S.P.Novikov from conductivity theory of normal metals~\cite{nov82}, where the asymptotic
behavior of unbounded connected components of the sections plays an important role,
with the surface being the Fermi surface of the metal and the plane direction being determined
by the external magnetic field.

\begin{figure}[hr]
\centerline{\includegraphics[scale=0.6]{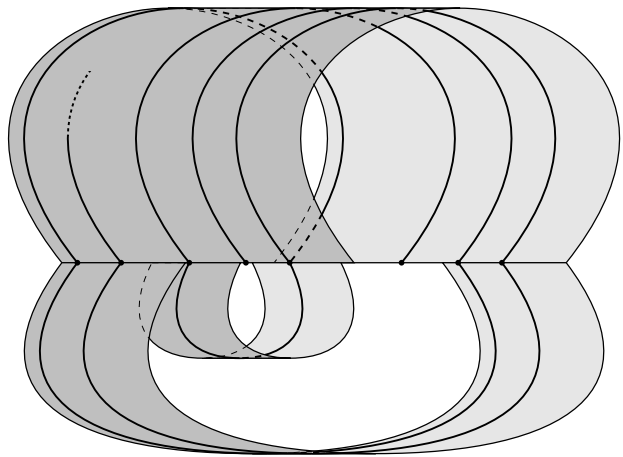}}
\caption{A foliated 2-complex made of three bands}\label{orbit}
\end{figure}

The discussed plane sections of 3-periodic surfaces can also be regarded as level sets
of a smooth quasi-periodic function $f:\mathbb R^2\rightarrow\mathbb R$ with three quasi periods. The
following is an example of such a function:
$$f(x,y)=\cos(x+a)+\cos(y)+\cos(\alpha x+\beta y),$$
where $\alpha,\beta\notin\mathbb Q$ and $a$ are some constants. Two types of the level line behavior, called later trivial and integrable,
have been well understood \cite{zor84,dyn93} and it was shown in~\cite{dyn93} that one of the
two cases occurs with probability one. In particular, it turned out that at most one level set of
a quasi-periodic function with 3 quasi-periods may be neither trivial nor integrable.
As discovered in \cite{dyn97} an exceptional level \emph{can} be of neither type,
in which case it is called chaotic. For example, only the zero level set of the function $f(x,y)$
above can be chaotic (then for any $a$) and this does occur
if $\alpha$ and $\beta$ are chosen appropriately.

In this problem, we deal with a measured oriented foliation on a closed surface embedded in the $3$-torus~$\mathbb T^3$,
so, one may be tempted to apply the theory of interval exchange maps~\cite{K75}.
However, this is unlikely to help here.
The reason is that closed $1$-forms that can appear as the restriction of a constant $1$-form
in~$\mathbb T^3$ to a level surface $M$ of a function satisfy certain restrictions. In particular,
the map $H_1(M,\mathbb Z)\rightarrow\mathbb R$ defined by integrating the $1$-form
along $1$-cycles has a large kernel. Structural theorems about
interval exchange maps usually assert something about almost all but not all
such maps. For example, unique ergodicity, which holds with probability one
for general irreducible interval exchange maps as shown by H.Masur~\cite{mas82} and W.Veech~\cite{v82},
becomes a rare guest here and may be observed \emph{only} in the chaotic case.

The main instrument that allows one to find and study chaotic examples in Novikov's problem
can be viewed as a particular case of an object that has been used
in the theory of dynamical systems, foliations, and geometric group theory
starting from early 1990's, a two-dimensional CW-complex equipped with a measured
foliation having finitely many singularities. A fundamental contribution to the theory
of such foliated complexes was made by E.Rips in his unpublished work
based on some ideas of G.Makanin and A.Razborov, see~\cite{bf95,glp94}.

A basic example of a foliated 2-complex is constructed as follows. Take an oriented
closed interval $D\subset\mathbb R$ and three disjoint \emph{bands}
$B_i=[a_i,b_i]\times[0,1]\subset\mathbb R^2$, $i=1,2,3$, foliated by vertical arcs
$\{x\}\times[0,1]$, and glue each of the horizontal sides $[a_i,b_i]\times\{0\}$, $[a_i,b_i]\times\{1\}$ of every band
to a sub-arc of~$D$ (which is foliated by points)
by an orientation preserving isometry, see Fig.~\ref{orbit}. After the glueing each leaf of the foliation will be
a $1$-dimensional simplicial complex composed of vertical arcs from the bands.

The obtained foliation depends on finitely many parameters, namely, the widths $(b_i-a_i)$ of the bands,
the length of $D$, and the gluing maps. By applying the general
theory, one can show that one of the following
simple cases occurs with probability one:
\begin{enumerate}
\item all leaves are compact;
\item every leaf is quasi-isometric to the plane $\mathbb R^2$, and leaves are assembled together in a way
equivalent (in a natural sense) to a foliation of $\mathbb T^3$ by parallel planes;
\item every non-compact leaf is quasi-isometric to the line $\mathbb R$, and non-compact leaves are assembled together in a way
equivalent to an irrational winding of $\mathbb T^2$.
\end{enumerate}
In the problem on plane sections of a 3-periodic surface, cases (1) and (2) correspond to the trivial behavior, and case (3) to the integrable one.

However, a completely different structure of leaves may occur under a special
choice of parameters. Namely, it is possible that every leaf of the foliation
is everywhere dense and is quasi-isometric to an infinite tree that has arbitrarily large branches.
(In the particular case of three bands one can actually omit `quasi-isometric to'.)
As follows from \cite{bf95} this property of leaves is a characterization of foliations of thin type,
though the definition is quite different. The first example of a foliated $2$-complex
of thin type was constructed by G.Levitt~\cite{l93a}.

Each foliation of thin type in the example above gives rise to a 3-periodic surface
whose plane sections of a fixed direction are chaotic, see~\cite{dyn08}.

The construction above is an example of a so called union of bands, which, in turn,
is a particular case of a band complex~\cite{bf95}. Band complexes provide a convenient
framework to study general measured foliated $2$-complexes, but this is not
the only possible model.
In~\cite{dyn97} $2$-complexes of different kind were used (they had a single $2$-cell), and the relation
to general theory was not noticed. However, the procedure
described there is essentially equivalent to a particular case of the Rips
machine after an appropriate translation.

The theory of band complexes is equivalent to that of systems of partial isometries of the line,
which were introduced in~\cite{glp94}.

One of the first things
we would like to know about chaotic sections is the number of connected
components (all of which are unbounded by construction). This problem is
reduced to the following question about the corresponding foliated 2-complex:
how many topological ends does a typical leaf have? A single topological
end would imply a single connected component of a typical section, and two topological ends
would imply infinitely many components.

It was noted by M.Bestvina and M.Feighn in \cite{bf95} and D.Gaboriau in~\cite{g96}
that all but finitely many leaves of a band complex
of thin type are quasi-isometric to infinite trees with at most two topological
ends, and shown that 1-ended and 2-ended leaves are always present
and, moreover, there are uncountably many leaves of both kinds.
In principle, this leaves three possibilities:
\begin{description}
\item[1-ended leaves win]
the union of 2-ended leaves has zero measure;
\item[2-ended leaves win]
the union of 1-ended leaves has zero measure;
\item[draw]
none of the above, in which case the foliation is not ``uniquely ergodic''.
\end{description}

A few examples have been examined in this respect \cite{bf95,s1,c10}, $1$-ended leaves won
in all of them. As shown in \cite{bf95,g96} the union of $1$-ended leaves
is a $G_\delta$ subset of the band complex, so, it was considered
possible that $1$-ended leaves would win in general.
In the present paper we show that this is not the case,
by constructing explicitly a foliated $2$-complex
of thin type in which $2$-ended leaves win. The complex can still be made of three
bands as shown in Fig.~\ref{orbit}, so it has respective implications for plane
sections of some $3$-periodic surfaces.

\begin{remark}
However, our example is not yet quite satisfactory for physical applications
mentioned in the very beginning of the paper.
All Fermi surfaces that may arise in physics obey a central symmetry, which translates
into a symmetry of the corresponding foliated 2-complex. So, the foliation
must be invariant under an involution that
flips the orientation of the interval~$D$. Constructing a symmetric
example in which 2-ended leaves win does not seem impossible but
looks harder and has not yet been tried.
\end{remark}

We show also that the reason why 1-ended leaves won in the previously tested
cases is the self-similarity of the complexes (to be defined below).
The point is that constructing a
foliated 2-complex of thin type is rather tricky, and self-similarity
provides the easiest way to proof the thinness.

We prove below in quite
general settings that self-similarity necessarily implies that 1-ended
leaves win.

Thierry Coulbois drew our attention that this implication can be deduced from his
result~\cite{c10} and
a result of M.Handel and L.Mosher~\cite{hm07} on parageometric automorphism
of free groups (provided we show that self-similarity in our sense implies
full irreducibility of the corresponding automorphism, which is likely to be true). He also pointed out to us that the limit set of
the repelling tree \cite{c10} is directly related
to the set of 2-ended leaves, and suggested that we can not only
say that the union of 2-ended leaves has zero measure but
also estimate the Hausdorff dimension of the set of points
that separate the leaves passing through them into two infinite parts.
We provide below a self-contained geometrical argument for the win
of $1$-ended leaves in the self-similar case.

The ``draw'' case is yet to be discovered. It seems unlikely that three bands would suffice for that,
but, for a larger number of bands, the non-uniqueness of the invariant transversal
measure, which is a necessary condition for a draw, has already been observed due to R.Martin \cite{mar97}.

The paper is organized as follows. Section~\ref{section-prelim} consists mostly of definitions.
In Section~\ref{section-self-similar} we show that self-similarity implies that
almost all leaves are $1$-ended. In Section~\ref{section-example} we construct an example
of a measured foliated $2$-complex in which almost all leaves are $2$-ended.

\subsection*{Acknowledgements} We thank Thierry Coulbois for fruitful discussions. The first
author is supported in part by the Russian Foundation for Fundamental Research
(grant no.~14-01-00012). The second author gratefully acknowledge the support of
ERC Starting Grant ``Quasiperiodic''.

\section{Preliminaries}\label{section-prelim}
\subsection{Foliated $2$-complexes}
Locally finite $2$-dimensional
CW-complexes homeomorphic to a simplicial complex will be called \emph{$2$-complexes} for short.

Foliations that we consider will be not only transversely measured but also transversely oriented,
so, we will speak in terms of closed $1$-forms.

\emph{A closed $1$-form} $\omega$ on a $2$-complex $X$ is a family of closed $1$-forms on
the closures of
all cells of~$X$ provided that these $1$-forms agree on intersections. We additionally require that
every $2$-cell of $X$ admits a parametrization $f:P\rightarrow X$
with a convex polygon $P\subset\mathbb R^2$ such that
$f^*(\omega)$ coincides with $dx|_P$, where $x$ is the first coordinate in $\mathbb R^2$.

A closed $1$-form $\omega$ on $X$ defines \emph{a foliation with singularities},
denoted $\mathcal F_\omega$ whose \emph{leaves} are maximal path-connected
subsets $L\subset X$ such that
for any arc $\alpha\subset L$ the restriction $\omega|_\alpha$ is identically zero.
The foliation $\mathcal F_\omega$ comes with \emph{transverse measure},
which is a functional on paths in $X$ defined by integrating~$|\omega|$.

A point $p\in X$ is \emph{a regular point of $\mathcal F_\omega$} if the restriction of $\mathcal F_\omega$
to some open neighborhood of $p$ is a trivial fiber bundle over an open interval, and \emph{a singular point}
otherwise. Clearly, singular points may occur only at vertices of $X$ and $1$-cells the restriction of
$\omega$ to which vanishes.
The closure of such a $1$-cell will be referred to as \emph{a vertical edge of $X$} if
there is a single $2$-cell attached to it.
A connected component of the set of singular points of $F_\omega$ will be called
\emph{a singularity} of $F_\omega$.

A leave of $\mathcal F_\omega$ is \emph{singular} if it contains a singular point, and \emph{regular}
otherwise.
For a singular leaf $L$,
we denote by $\sing(L)$ the set of singular points of $L$.

A parametrized path $\gamma:I\rightarrow X$, where $I$ is an interval, is \emph{transverse} to $\mathcal F_\omega$
if $\gamma^*\omega=dt$, where $t$ is a monotonic smooth function on $I$. If, additionally, $\gamma$ is injective,
then its image $\sigma$ is called \emph{a transversal arc} and the the value $|\int_I\gamma^*\omega|$,
which is clearly defined solely by $\sigma$, is said to be \emph{the weight} of $\sigma$. It is denoted by $|\sigma|$.

We implicitly assume that every $2$-complex $X$ that we consider comes with
a metric (that agrees with the topology of $X$), the particular choice of which will not be important.
Every leaf $L$ of $\mathcal F_\omega$ is then also
endowed with a metric $d_L$, which is defined as follows.
The distance $d_L(p,q)$ between any two points $p,q\in L$
is defined as the length of the shortest path $\gamma\subset L$ connecting $p$ and $q$.

\subsection{Union of bands}
\emph{A band} is a (possibly degenerate) rectangular $B=[a,b]\times[0,1]\subset\mathbb R^2$, $a\leqslant b$, endowed with the $1$-form $dx$,
where $x$ is the first coordinate in $\mathbb R^2$.
The horizontal sides $[a,b]\times\{0\}$ and $[a,b]\times\{1\}$ are called \emph{the bases} of the band, any
arc of the form $\{c\}\times[0,1]$ with $c\in[a,b]$ is called \emph{a vertical arc}. The band is \emph{degenerate} if
$a=b$. The value $(b-a)$ is called \emph{the width} of the band and denoted $|B|$.

\emph{A union of bands} is a $2$-complex $X$ endowed with a closed $1$-form $\omega$ obtained from a union~$D$
of pairwise disjoint closed (possibly degenerate to a point) intervals
of $\mathbb R$, called \emph{the support multi-interval of~$X$},
and several pairwise disjoint bands $B_i=[a_i,b_i]\times[0,1]$ by gluing each base
of every band isometrically and preserving the orientation to a closed subinterval of $D$. The form $\omega$ is
the one whose restriction to each band is $dx$, so, we keep using notation $dx$ for it.

\begin{remark}
Our definition of a union of bands is less general than the one in~\cite{bf95}, where $D$ is allowed to
be an arbitrary $1$-dimensional simplicial complex and preserving orientation is not
required for the gluing maps. If we omit the preserving orientation requirement
the class of unions of bands that we consider
will be precisely that of suspensions of system of partial isometries of the line as defined in~\cite{glp94}.
\end{remark}

A union of band is \emph{non-degenerate} if it does not contain a degenerate band or an isolated point.
\emph{The solid part} of a union of bands $X$ is the union of bands obtained from $X$
by removing all degenerate bands and isolated points of $D$.

A vertical side of any non-degenerate band in $X$ is called \emph{a vertical edge of $X$}.
The vertical edge $\{a\}\times[0,1]$ of a band $[a,b]\times[0,1]$
if \emph{a left edge of $X$}, and $\{b\}\times[0,1]$ \emph{a right edge of $X$}.

For a union of bands $X$, we denote by $|X|$ the sum of the widths of all the bands.
For a multi-interval $D\subset\mathbb R$ we denote by $|D|$ the sum
of the lengths of all intervals in $D$. The difference $|X|-|D|$ will be called
\emph{the excess} of $X$ and denoted $\exc(X)$.

Let $Y_1$ and $Y_2$ be unions of bands with support multi-intervals $D_1$ and $D_2$, respectively.
We say that they are \emph{isomorphic}
if there is a homeomorphism $f:Y_1\rightarrow Y_2$ such that we have $f^*(dx)=dx$.
If, additionally, $Y_1$ has minimal possible number of bands among unions of bands
isomorphic to $Y_2$ and we have $f(D_1)\subset D_2$, then the image $f(B)$ of any band $B$ of $Y_1$ is called
\emph{a long band of $Y_2$}.

It is easy to see that for isomorphic unions of bands $Y_1$, $Y_2$ we have $\exc(Y_1)=\exc(Y_2)$.

\emph{An enhanced union of bands} is a union of bands $Y$ together with a non-trivial assignment
of a non-negative real number to each band. This number is called \emph{the length of the band}.
A band of width $w$ and length $\ell$ is said to have \emph{dimensions $w\times\ell$}.
`Non-trivial' means that at least one of the lengths is positive. \emph{The length of a long band} $B$
is defined as the sum of the lengths of all bands contained in $B$.
Two enhanced unions of bands are \emph{isomorphic} if they are isomorphic as
unions of bands and their respective long bands have the same length.

Let $(X,\omega)$ be an arbitrary finite $2$-complex endowed with a closed $1$-form, $Y$ a union of bands.
We say that $Y$ is \emph{a model} for $(X,\omega)$ if there is a continuous surjective map $\theta:Y\rightarrow X$,
called \emph{the projection} of this model, such that:
\begin{enumerate}
\item
$\theta^*\omega=dx$,
\item
the preimage of any leaf of $\mathcal F_\omega$ is a single leaf of $\mathcal F_{dx}$, and
\item
$\theta$ is injective on an open subset $U\subset Y$ whose complement $Y\setminus U$
is contained in finitely many leaves of $\mathcal F_{dx}$.
\end{enumerate}
One can show that every finite $2$-complex
endowed with a closed $1$-form has a model. It is also clear that being a model is a transitive relation.

\subsection{First return correspondence}\label{first-return}
Here we introduce an analogue of the first return map of a dynamical system.

Let $I$ be a closed interval. \emph{A correspondence on $I$} is an equivalence relation on
the Cartesian product $I\times\{{+},{-}\}$.
The exchange ${+}\leftrightarrow{-}$ induces an involution on
the set of correspondences, which we call \emph{a flip}.

Let $\sim_1$, $\sim_2$ be correspondences on intervals $I_1$, $I_2$, respectively,
and let $f:x\mapsto\lambda x+c$ be the unique affine map with $\lambda>0$ such that $f(I_1)=I_2$.
Denote by $\sim_1'$ be the correspondence on $I_1$ defined by the rule:
$(x_1,\epsilon_1)\sim_1'(x_2,\epsilon_2)$ if and only if
$(f(x_1),\epsilon_1)\sim_2(f(x_2),\epsilon_2)$.
If $\sim_1$ and $\sim_1'$ coincide or are obtained from each other by a flip
we say that $\sim_1$ and $\sim_2$ are \emph{similar}.

Let $(X,\omega)$ be a 2-complex with a closed $1$-form, and let
$\sigma\subset X$ be a closed transversal arc such that the interior of $\sigma$
is contained in the interior of a single $2$-cell, and each of the endpoints of $\sigma$
is contained in the interior of the same cell or in the interior of a vertical edge of $X$.
So, the arc $\sigma$ locally cuts every leaf
into two parts. Let us choose a coorientation on $\sigma$, thus assigning
to one of those parts the plus sign, and to the other the minus sign.
We identify $\sigma$ with an interval of $\mathbb R$ by a map $f:\sigma\rightarrow\mathbb R$
such that $df=\omega|_\sigma$.

Now denote by $\sim_\sigma$ the correspondence on $\sigma$ defined as follows:
we have $(p_1,\epsilon_1)\sim_\sigma(p_2,\epsilon_2)$, $p_1,p_2\in\sigma$,
$\epsilon_1,\epsilon_2\in\{{+},{-}\}$, if and only if there is an arc $\alpha\subset X$ such that:
\begin{enumerate}
\item $\alpha$ is contained in a single leaf of $\mathcal F_\omega$;
\item $\alpha\cap\sigma=\partial\alpha=\{p_1\}\cup\{p_2\}$;
\item $\alpha$ approaches $p_i$ from the $\epsilon_i$-side, $i=1,2$.
\end{enumerate}
We say that $\sim_\sigma$ is \emph{the first return correspondence induced by $\mathcal F_\omega$} (or by $\omega$)
\emph{on $\sigma$}. It is uniquely defined up to a flip.

\subsection{Cutting and splitting}\label{cuttingsubsection}
Let $X$ be a union of bands with support multi-interval $D$, and let $B=[a,b]\times[0,1]$ be one of its bands.
Denote by $g_0$ and $g_1$ the gluing maps $[a,b]\times\{0\}\rightarrow D$ and
$[a,b]\times\{1\}\rightarrow D$, respectively. Let us pick $c\in[a,b]$ and choose $\varepsilon>0$
so that the strip $(b,b+\varepsilon]\times[0,1]$ is disjoint from all bands.

Suppose $c\in(a,b)$.
Removing the band $B$ from $X$ and replacing it with two bands $[a,c]\times[0,1]$ and $[c+\varepsilon,b+\varepsilon]\times[0,1]$
with the gluing maps $g_0|_{[a,c]\times\{0\}}$, $g_1|_{[a,c]\times\{1\}}$ for the first one and
$g_0|_{[c,b]\times\{0\}}\circ s_{-\varepsilon}$, $g_1|_{[c,b]\times\{1\}}\circ s_{-\varepsilon}$,
where $s_{-\varepsilon}(x,y)=(x-\varepsilon,y)$,
for the latter, will produce another union of bands $X'$. We say that $X'$ is obtained from~$X$
by \emph{cutting along the vertical arc $\{c\}\times[0,1]$}, see Fig.~\ref{cutting}.
\begin{figure}[ht]
\centerline{\includegraphics{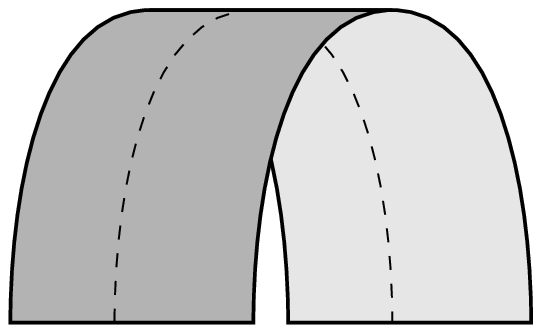}\put(-134,22){\begin{rotate}{-100}{\Large\ScissorLeftBrokenTop}\end{rotate}}\quad
\raisebox{45pt}{$\rightarrow$}\quad
\includegraphics{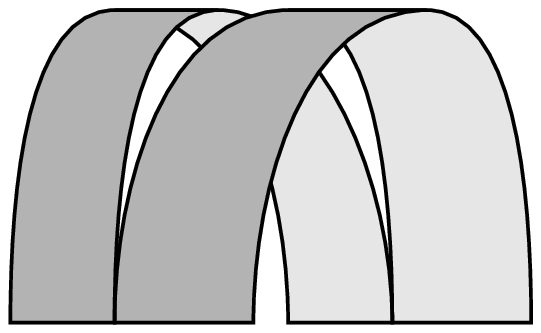}}
\caption{Cutting along a vertical arc}\label{cutting}
\end{figure}

If $c\in\{a,b\}$, then by cutting along the vertical arc $\{c\}\times[0,1]$ we mean the trivial operation,
which leaves $X$ unchanged.

Let $[a,b]$ be a component of $D$, and suppose that there is a point $c\in(a,b)$ to which
no interior point of a base is attached, but some endpoint of a base is attached. Then we can \emph{split $X$ at $c$} (accordingly, $c$
will be called \emph{a splitting point for $X$}), which means the following.
Replace $[a,b]$ with $[a,c]\cup[c+\varepsilon,b+\varepsilon]$ for a small enough $\varepsilon$ and modify
accordingly every attaching map $g$ of each base such that the image of $g$ is contained in $[c,b]$:
$g\mapsto g+\varepsilon$. Additionally, connect the points $c$ and $c+\varepsilon$ by a new degenerate band,
see Fig.~\ref{splitting}.
\begin{figure}[ht]
\centerline{\includegraphics{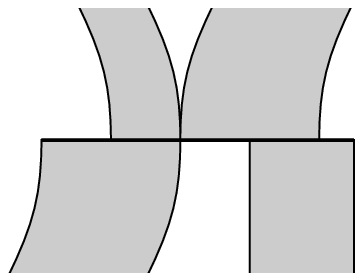}\put(-63.5,42){\begin{rotate}{-80}{\Large\ScissorLeftBrokenBottom}\end{rotate}}\quad
\raisebox{45pt}{$\rightarrow$}\quad\includegraphics{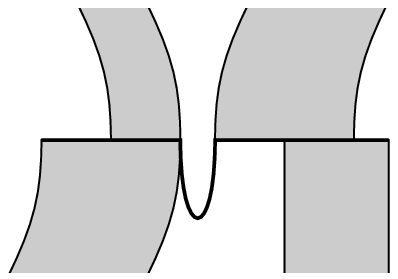}}
\caption{Splitting}\label{splitting}
\end{figure}

Let $\Gamma\subset X$ be a finite graph whose edges are vertical arcs. By \emph{cutting $X$ along $\Gamma$}
we mean producing a new union of bands $X'$ by applying the following two steps:
\begin{enumerate}
\item cut $X$ along every edge of $\Gamma$;
\item split the obtained union of bands at any splitting point.
\end{enumerate}

The union of bands $X'$ obtained from $X$ by cutting along a graph consisting
of vertical arcs is obviously a model for $X$. The model projection $X'\rightarrow X$
is obtained by ``gluing back'' the cuts and contracting every new degenerate band to a point.

Let $B$ be a band of $X$. By \emph{subdividing $B$} we mean replacing $B$ with two new bands $B'$, $B''$ such
that $|B'|=|B''|=|B|$, the $0$-base of $B'$ and the $1$-base of $B''$ are attached to where the $0$-base and $1$-base
of $B$ were, and the other two bases are attached to a newly added to $D$ disjoint closed interval of length~$|B|$.
The new interval in $D$ is called \emph{the subdivision arc}.

Let $\sigma$ be a connected component of $D$. By \emph{cutting $X$ along $\sigma$} we mean detaching
all bands whose bases are glued to $\sigma$ and attach them to newly added to $D$ pairwise disjoint closed intervals.

\subsection{Block decomposition}

The idea behind the block decomposition is to describe the first return correspondence induced by a union
of bands in a manner similar to describing the first return map of a transverse measured and transverse
oriented foliation on a surface by an interval exchange.

Let $\sigma$ be \emph{a horizontal arc} in some band $B=[a,b]\times[0,1]$ of a union of bands $X$,
i.e. an arc of the form $[c,d]\times\{h\}$, with $[c,d]\subset[a,b]$, $h\in(0,1)$.
Denote by $\sigma'$ and $\sigma''$ the arcs in $D$ to which the portions
$[c,d]\times\{0\}$ and $[c,d]\times\{1\}$ of the bases are attached.

Let $L$ be a singular leaf of $\mathcal F_{dx}$. The closure $N$ of a connected component
of $L\setminus(\stackrel\circ\sigma\cup\sing(L))$, where~$\stackrel\circ\sigma$ is the interior of $\sigma$,
is said to be \emph{a singularity extension toward $\sigma$} if 
$N$ has finite diameter and $N$ has a non-empty intersection with $\sing(L)$.
Clearly there are only finitely many singularity extensions toward~$\sigma$.

By \emph{the block decomposition of $X$ induced by $\sigma$} or simply \emph{the $\sigma$-decomposition} we
call the band complex~$X_\sigma$ obtained as follows. First, we cut $X$ along the union of vertical arcs
$\{c\}\times[0,1]$ and $\{d\}\times[0,1]$. Then we subdivide the band that is now
attached to $\sigma'$ and $\sigma''$. We abuse notation by letting $\sigma$
denote the subdivision arc. Finally, we cut the obtained band complex along all
singularity extensions toward~$\sigma$.

A connected component of the union of bands obtained from the solid part of $X_\sigma$
by cutting along $\sigma$  is called \emph{a block} of $X_\sigma$. 
If a block $\mathcal B$ has the form of a trivial fibration over a closed interval with fibre a finite graph,
it is called \emph{a product block}.

\emph{An arm} of a block $\mathcal B$ is a pair $(B,\eta)$, where $B$ is a long band of $\mathcal B$
and $\eta$ is a base of $B$ such that no other band is attached to $\eta$. An arm $(B,\eta)$
is \emph{free} if the image of $\eta$ in $X_\sigma$ is disjoint from $\sigma$. Otherwise
it is said to be \emph{bound to $\sigma$}. The image of $\eta$ in $\sigma$
will then be called \emph{the binding arc} of the arm $(B,\eta)$.

Unless $\mathcal B$ is a single long band the arc $\eta$ is
uniquely defined by $B$, so, we may sometimes speak of $B$ as an arm.

The gluing back map $X_\sigma\rightarrow X$ will be denoted by $\theta_\sigma$. If $X$ is non-degenerate,
then $X_\sigma$ is a model for~$X$ with model projection $\theta_\sigma$.

The following statement is obvious.

\begin{proposition}\label{obviousprop}
Let $Y$ and $Y'$ be unions of bands, $\sigma\subset Y$ and $\sigma'\subset Y'$ horizontal arcs.
Let $f:\sigma\rightarrow\sigma'$ be a homeomorphism such that $f^*(dx)=dx|_\sigma$.
Suppose that all blocks of $Y_\sigma$ and $Y'_{\sigma'}$ are product blocks
and that the first return correspondences induced on $\sigma$ and $\sigma'$ are similar.
Then, for any $p\in\sigma$, the leaf of~$Y$ passing through $p$ is quasi-isometric
to the leaf of $Y'$ passing through $f(p)$.
\end{proposition}

\subsection{Annulus-free complexes and compact leaves}
A $2$-complex $X$ with a closed $1$-from $\omega$ is \emph{annulus-free} if there is no embedded annulus $A\subset X$
such that $(A,\omega|_A)$ is foliated by circles. Equivalently: every regular leaf of $\mathcal F_\omega$ is simply connected.

To keep exposition simpler we restrict ourselves to considering only annulus-free foliations. Without this hypothesis
definitions become more involved, but nothing essential changes. Namely, if a foliation of thin type is not annulus-free, then
one can make it annulus-free, without changing other properties that we discuss, by
a number of operations like cuttings and subdivisions. For a detailed general account on the Rips machine, of which we consider
only a special case, the reader is referred to~\cite{bf95}.

Let $Y$ be a union of bands with support multi-interval $D$.
As follows from H.Imanishi's theorem~\cite{i79,glp94} there exists a (possible empty) finite union $E$ of subintervals of $D$ such that
every compact leaf of $\mathcal F_{dx}$ intersects~$E$ exactly once and $E$ is disjoint from
non-compact regular leaves. Its measure~$|E|$
is referred to as \emph{the measure of compact leaves of $\mathcal F_{dx}$}. It does not depend
on a particular choice of $E$. Moreover, this measure satisfies the following simple relation if the $2$-complex is
annulus-free.

\begin{proposition}\label{prop-compact-leaves-measure}
Let $Y$ be an annulus-free union of bands with support $D$. Then the measure
of compact leaves equals to the negative excess of $Y$: $|E|=-\exc(Y)$.
\end{proposition}

This claim is a reformulation of Proposition~6.1 from~\cite{glp94}.

A union of bands $Y$ with support multi-interval $D$ is called \emph{balanced}
if $\exc(Y)=0$. It follows from Proposition~\ref{prop-compact-leaves-measure} that a balanced
union of bands is annulus-free if and only if it does not have a compact leaf.

\subsection{Foliated $2$-complexes of thin type}

Let $Y$ be a union of bands with support multi-interval $D$. \emph{A~free arc} of~$Y$ is a maximal open interval $J\subset D$
such that it is covered by one of the bases of bands, and all other bases are disjoint from $J$.

Let $J$ be a free arc and $B=[a,b]\times[0,1]$ be the band one of whose bases covers $J$
under the attaching map. Let $(c,d)\subset[a,b]$ be the subinterval such that $(c,d)\times\{0\}$
or $(c,d)\times\{1\}$ is identified with $J$ in $Y$. Let $Y'$ be the union of bands obtained from $Y$ by removing
$J$ from $D$, and $(c,d)\times[0,1]$ from $B$ thus replacing $B$ with two smaller bands $[a,c]\times[0,1]$ and
$[d,b]\times[0,1]$
whose bases are attached to $D$ by the restriction of the attaching maps for the bases of $B$.
If this produces an isolated point of $D$ such that only one degenerate band is attached to it
(which may occur if $a=c$ or $b=d$), the point and the band are removed.
We then say that $Y'$ is obtained from $Y$ by \emph{a collapse from a free arc}, see Fig.~\ref{collapse}.
\begin{figure}[ht]
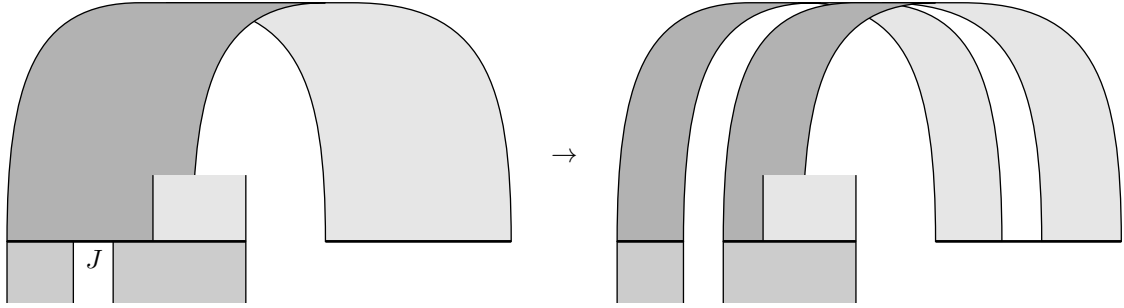

\centerline{\includegraphics{collapse1.eps}\put(-166,15){$J$}\quad
\raisebox{55pt}{$\rightarrow$}\quad
\includegraphics{collapse2.eps}}
\caption{Collapse from a free arc}\label{collapse}
\end{figure}

It is easy to see that the excess does not change under this operation, $\exc(Y')=\exc(Y)$.

If $Y$ is an enhanced union of bands, than $Y'$ is also regarded enhanced.
The length of every band $B'$ of $Y'$ is set to that of the band in $Y$ containing $B'$.

An annulus-free union of bands $Y$ is said to be \emph{of thin type} if the following two conditions hold:
\begin{enumerate}
\item
every leaf of the foliation $\mathcal F_{dx}$ is everywhere dense in $Y$;
\item there is an infinite sequence $Y_0=Y,Y_1,Y_2,\ldots$ in which every
$Y_i$ is a union of bands obtained from $Y_{i-1}$ by a collapse from a free arc.
\end{enumerate}
Such a sequence as well as any infinite subsequence $Y_0,Y_{i_1},Y_{i_2},\ldots$
with $i_1<i_2<\ldots$
is said to be \emph{produced by the Rips machine} or \emph{a Rips sequence} for short.

\begin{remark}
The reader is warned that this definition does not insist on the foliation $\mathcal F_{dx}$
being minimal in the sense of~\cite{bf95}, but this plays no role here.
\end{remark}

An annulus-free finite $2$-complex $X$ endowed with a closed $1$-form is said to be \emph{of thin type} if
some (and then any) model of $X$ is of thin type.

We will need the following two elementary facts about the Rips machine.

\begin{proposition}\label{elem1}
Let $Y$ be an annulus-free union of bands of thin type,
$Y\supset Y_1\supset Y_2\supset\ldots$ and $Y\supset Y_1'\supset Y_2'\supset\ldots$
two sequences produced by the Rips machine. Then for any $k>0$ there exists
$l$ such that $Y_l\subset Y'_k$.
\end{proposition}

\begin{proof}
As shown in \cite{bf95} the intersection $\bigcap_{i=1}^\infty Y_k'$
is nowhere dense. It means, that if $J\subset Y_k$ is a free arc, then
eventually a portion of $J$ will be removed by the Rips machine.
This may occur only as a result of a collapse from $J$ or a larger arc.
This means that if the Rips machine \emph{can} remove some portion of~$Y_k$,
then it \emph{will} eventually do so.
\end{proof}

We call a horizontal arc $\sigma\subset Y$ in a union of bands \emph{inessential} if,
for a Rips sequence $Y\supset Y_1\supset Y_2\supset\ldots$, the intersection
$\sigma\cap Y_k$ is finite for all but finitely many $k$.
We will express this by saying that $\sigma$ is \emph{removed} by the Rips machine \emph{almost completely}.
It follows from Proposition~\ref{elem1} that the choice of the
Rips sequence does not matter in this definition.

If a horizontal arc is not inessential it is called \emph{essential}.

The same terminology applies if $\sigma$ is a subinterval of the support multi-interval of $Y$.

\begin{proposition}\label{essential->product}
Let $Y$ be an annulus-free union of bands of thin type, $\sigma$ an essential
horizontal arc. Then all blocks of $Y_\sigma$ are product ones.
\end{proposition}

\begin{proof}
Suppose otherwise, that some block $\mathcal B$ of $Y_\sigma$ is
not a product block. Then the foliation $\mathcal F_{dx}$ on $\mathcal B$
does not have compact leaves and is still annulus-free, which implies
$\exc(\mathcal B)=0$ as follows from Proposition~\ref{prop-compact-leaves-measure}.
Let $\eta_1,\eta_2,\ldots,\eta_r$ be the binding arcs of $\mathcal B$.
They are pairwise disjoint as otherwise
the image of $\mathcal B$ in $Y_\sigma$ would have positive excess
and not be annulus-free. So, we may not distinguish
between $\mathcal B$ and its image in $Y_\sigma$.

\begin{lemma}
The arcs $\eta_1,\ldots,\eta_r\subset Y$ are inessential.
\end{lemma}

\begin{proof}
Let us choose one of them, $\eta=\eta_i$, say, and consider
the block decomposition $Y_\eta$. Since there is no compact leaves
in $\mathcal B$, there will be no singularity extensions toward $\eta$
that passes through $\theta_\sigma(\mathcal B)$. It means that $Y_\eta$
will have a possibly larger block $\mathcal B'$ such that it will be not a product block
and $\eta$ will be one of its binding arcs.

Let $\mathcal B_1,\ldots,\mathcal B_s$ be the other blocks of $Y_\eta$.
Cutting $Y_\eta$ along $\eta$ reduces the excess by $|\eta|$, so, we have
\begin{equation}\label{excesssum}
\sum_{j=1}^s\exc(\mathcal B_j)=-|\eta|.
\end{equation}
On the other hand, if $\xi_1,\ldots,\xi_t$ are all binding arcs of $\mathcal B_j$'s,
then
\begin{equation}\label{bindingsum}
\sum_{j=1}^t|\xi_j|\leqslant|\eta|.
\end{equation}
If $\mathcal B_j$ is a product block,
each its binding arc has weight $-\exc(\mathcal B_j)$. If $\mathcal B_j$
is not a product block, then $\exc(\mathcal B_j)=0$.
Any block must have at least one binding arc.

This together with~\eqref{excesssum} and \eqref{bindingsum}
imply that all $\mathcal B_j$, $j=1,\ldots,s$, are product blocks each
with a single binding arc, and their binding arcs cover $\eta$
without overlaps. Each such block is removed by the Rips machine
in finitely many steps, hence, $\eta$ considered as an arc in $Y_\eta$
will eventually be removed completely.
It is easy to see that this implies that $\eta$ will also be
removed almost completely by the Rips machine run on $Y$ instead of $Y_\eta$.
\end{proof}
Now run the Rips machine on $Y$ until all the arcs $\eta_1,\ldots,\eta_r$ are removed
almost completely. Let $Y'\subset Y$ be the obtained union of bands at this stage.

Since there are  no compact leaves in $\mathcal B$ the intersection
$\theta_\sigma(\mathcal B)\cap Y'$ still contains an unbounded portion of
any leaf $L\subset Y$. Since $\eta_1\cup\ldots\cup\eta_r$
is removed almost completely, only finitely many leaves can
``escape'' from $\theta_\sigma(\mathcal B)\cap Y'$ through $Y'$ and intersect
$\sigma\setminus(\eta_1\cup\ldots\cup\eta_r)$. Hence, the subset
$\sigma\cap Y'$ is at most countable. But due to the nature of the Rips machine,
this intersection has form of a finite union of possibly degenerate
closed intervals, so, it should be either finite or uncountable. Thus, it is finite, which means that $\sigma$
is removed almost completely.
\end{proof}

\subsection{Trees}
\emph{An $\mathbb R$-tree} is a metric space $(T,d)$ in which any two points can be connected by a unique
simple arc, and any simple arc $\sigma\subset T$ is isometric to an interval of $\mathbb R$.
We omit `$d$' in notation unless the same $T$ is considered with different metrics.

An $\mathbb R$-tree is called \emph{a simplicial tree} or simply \emph{a tree} if it is homeomorphic to a locally finite 1-dimensional simplicial
complex.

For any point $p\in T$ of an $\mathbb R$-tree the connected components of $T\setminus\{p\}$
are called \emph{branches} of $T$ at $p$.
A point $p\in T$ is \emph{a branch point} if there are at least three branches of $T$ at $p$.

A point $p\in T$ of a simplicial tree is \emph{a terminal point} if there is just one edge attached to $p$.
Such an edge is also called \emph{terminal}.

A branch of a simplicial tree is \emph{inessential} if it's diameter is finite.
The complement to the union of all inessential branches of $T$ is called
\emph{the core} of $T$ and denoted $\core(T)$.

A simplicial tree is \emph{$k$-ended} if it has exactly $k$ topological ends.
Equivalently, $k$ is the maximal number of pairwise disjoint rays in $T$, where
\emph{a ray} is a subtree in $T$ isometric to $[0,\infty)$.
For example, $1$-ended trees are those that are not compact and have
empty core, and $2$-ended trees are those whose core is a line.

The terms `(inessential) branch' and `core' also apply when $T$ is an
arbitrary locally finite connected simplicial complex, using the same definition.

An $\mathbb R$-tree equipped with an action of a group $G$ by isometries is
called \emph{a $G$-tree}. The $G$-orbit of a pair $(p,\beta)$, where $p\in T$
is a branch point and $\beta$ is a branch of $T$ at $p$, is called
\emph{a branching direction of the $G$-tree $T$}.

A continuous map $\ori:T\rightarrow\mathbb R$ is \emph{an orientation map}
if for any $p,q\in T$ the distance $d(p,q)$ is equal to the total variation of
$\ori$ over the arc connecting $p$ and $q$. Two orientation maps
are \emph{equivalent} if their difference is constant. \emph{An orientation}
of an $\mathbb R$-tree is an equivalence class of orientation maps.
The orientation defined by an orientation map $\ori$ is denoted by $[\ori]$.
A $G$-tree is called \emph{oriented} if it is oriented as an $\mathbb R$-tree,
and the orientation is $G$-invariant.

An arc $\sigma\subset T$ in an $\mathbb R$-tree is \emph{monotonic} with respect to
the orientation defined by an orientation map~$\ori$ if the restriction of $\ori$ to $\sigma$
is injective.

An oriented $G$-tree $(T,d,[\ori])$ is \emph{self-similar} if for some $\lambda\ne1$, there is an orientation
preserving isometry $\Phi:(T,d,[\ori])\rightarrow(T,\lambda d,[\lambda\ori])$
and an automorphism $\phi$ of the group $G$ such that
$\phi(g)\cdot p=\Phi\bigl(g\cdot\Phi^{-1}(p)\bigr)$ for all $g\in G$, $p\in T$.
Such $\Phi$ is called \emph{a homothety} of the $G$-tree $T$, and $\lambda$
\emph{the shrink factor} of $\Phi$.

\subsection{Associated $\mathbb R$-tree}
Let $(X,\omega)$ be an annulus free finite 2-complex endowed with a closed $1$-form, $H$ the subgroup of
the fundamental group $\pi_1(X)$ normally generated by all loops contained in singular leaves,
and let $\widetilde X$ be the covering of $X$ corresponding to $H$. Denote by
$\proj_X$ the projection $\widetilde X\rightarrow X$ and by $G(X)$ the group
$\pi_1(X)/H$.

The foliation on $\widetilde X$
whose leaves are $\proj_X^{-1}(L)$, where $L$ is a leaf of $\mathcal F_\omega$
will denoted by $\widetilde{\mathcal F}_\omega$. We abuse notation by using the same letter $\omega$
for the preimage $\proj_X^*\omega$ of $\omega$ on $\widetilde X$.

Let $T$ be the set of leaves of $\widetilde{\mathcal F}_\omega$. We endow $T$ with the following pseudodistance $d$:
$$d(L_1,L_2)=\inf_\gamma\int_\gamma|\omega|,$$
where the infumum is taken over all paths $\gamma$ connecting a point from $L_1$ to a point from $L_2$.

As G.Levitt shows in~\cite{l93b} (in different terms) $d$ is actually a metric, and
$(T,d)$ is an $\mathbb R$-tree. One can also see that it is naturally
an oriented $G$-tree with an orientation map given by
$$\ori(L)=\int_p^q\omega,\text{ where }p\in\widetilde X\text{ is a fixed point, and }q\in L.$$
Thus constructed oriented $G(X)$-tree $(T,d,[\ori])$ is said to be \emph{associated with} $(X,\omega)$.
We also say that $(X,\omega)$ \emph{resolves} this oriented $G$-tree, where $G=G(X)$.

\section{Self-similar foliated $2$-complexes}\label{section-self-similar}

\begin{theorem}\label{matintheorem1}
Let $(X,\omega)$ be an annulus-free $2$-complex with a closed $1$-form.
If it is of thin type and the associated oriented $G(X)$-tree is self-similar, then
the union of leaves of the foliation $\mathcal F_\omega$ defined by $\omega$ that
are not $1$-ended trees has zero measure. Moreover, the union of the cores
of all leaves of $\mathcal F_\omega$ has Hausdorff dimension in the interval $(1,2)$.
\end{theorem}

The proof of this theorem occupies the rest of this section. The plan is as follows. First, we
reduce the general case to a situation when several additional technical assumptions hold.
Then, under those additional assumptions, we show that self-similarity of
the associated $\mathbb R$-tree implies that of the first return correspondence induced on
a horizontal arc. Finally, we use the self-similarity of the first return correspondence
to construct a ``periodic'' Rips sequence,
which make the argument used for particular examples \cite{bf95,s1} work in the general case.

All $2$-complexes with a closed $1$-from are assumed to be annulus-free
in the sequel.

\subsection{More assumptions}

For a $2$-complex endowed with a closed $1$-form $(X,\omega)$ of thin type
we denote by $\mathcal C(X)$ the union of cores of all leaves of $\mathcal F_\omega$.

Let $(Y,dx)$ be a model for $(X,\omega)$, and let $\theta:Y\rightarrow X$ be the model
projection. The following are easy consequences of the model definition:
\begin{lemma}
The projection $\theta$ induces an isomorphism $G(Y)\rightarrow G(X)$.

$(Y,dx)$ resolves the same oriented $G(X)$-tree $T$ as $(X,\omega)$ does.

$\mathcal C(X)$ is a subset of $\theta(\mathcal C(Y))$, and the difference
$\theta(\mathcal C(Y))\setminus \mathcal C(X)$ can be covered by finitely many arcs.
\end{lemma}

This means that it is enough to prove the assertion of the theorem for $Y$ instead of $X$
or for any other $2$-complex for which $Y$ is a model.

Let $Y_0=Y\supset Y_1\supset Y_2\supset\ldots$ be a Rips sequence.

\begin{lemma}
The inclusion $Y_k\hookrightarrow Y$ induces an isomorphism $G(Y_k)\rightarrow G(Y)$.
The union of bands $Y_k$ resolves the same oriented $\mathbb R$-tree as
$Y$ does.

The ``output'' $\bigcap_{k=0}^\infty Y_k$ of the Rips machine does not depend on the
particular choice of the sequence and coincides with $\mathcal C(Y)$.
For all $k$ we have $\mathcal C(Y_k)=\mathcal C(Y)$.
\end{lemma}

\begin{proof}
It follows easily from the definition of a collapse from a free arc that
the inclusion $Y_{k+1}\hookrightarrow Y_k$ induces an isomorphism $G(Y_{k+1})\rightarrow G(Y_k)$
and an isometry of the associated $G$-trees with $G=G(Y_{k+1})=G(Y_k)$,
which implies the first two claims of the Lemma by induction.

The last two claims follow from the fact that the Rips machine eventually remove
any inessential branch of any leaf, see~\cite{bf95}.
\end{proof}

Thus, it would suffice to prove the theorem for any of $Y_k$, $k=0,1,2,\ldots$, chosen from
an arbitrary sequence produced from $Y$ by the Rips machine. If there is an assertion
that holds for all but finitely many $Y_k$ we may assume from the beginning that it holds
for $Y$.

\begin{definition}\label{reduced.u.o.b}
Let $L$ be a singular leaf of a union of bands $Y$. Denote by $L^0$ the union of degenerate bands
in $L$, by $L^-$ the union of right vertical edges of $Y$ contained in $L$,
by $L^+$ the union of left ones, and by $L^*$ the closure of $L\setminus\sing(L)$.
We call the leaf $L$ \emph{reduced} if the following conditions hold:
\begin{enumerate}
\item
$H_1(L,\sing(L))=0$;
\item
if $p$ is a terminal point of $L^-$ or $L^+$, or a point of $L^-\cap L^+$, then $p$ belongs to
an unbounded component of $L^*$;
\end{enumerate}
A union of bands is \emph{reduced} if every singular leaf of $Y$ is reduced.
\end{definition}

\begin{lemma}
A model $Y$ for $(X,\omega)$ and
a Rips sequence $Y_0=Y\supset Y_1\supset Y_2\supset\ldots$ can
be chosen so that
all but finitely many $Y_k$'s are reduced.
\end{lemma}

\begin{proof}
Pick an arbitrary model $Y$ for $(X,\omega)$. Let $L$ be a singular leaf
of $Y$. Suppose that $H_1(L,\sing(L))$ is non-trivial. Then there is an arc
$\alpha$ in $L^*$ such that $\alpha\cap\sing(L)=\partial\alpha\subset\sing(L)$.
Cutting $Y$ along $\alpha$ (see subsection~\ref{cuttingsubsection})
reduces the dimension of $H_1(L,\sing(L))$
by one. So, by performing such cuttings we may ensure that condition~(1) of Definition~\ref{reduced.u.o.b}
holds for $Y$.

A collapse from a free arc cannot introduce new generators to $H_1(L,\sing(L))$, so, if condition~(1)
holds for $Y$, it will hold for all members of any
Rips sequence $Y=Y_0\supset Y_1\supset Y_2\supset\ldots$.

Let again $L$ be a singular leaf of $Y$. If $L^*$ has a bounded connected component $\Gamma$
that contains a point from $L^+\cap L^-$ or a terminal point of $L^+$ or $L^-$ we cut
$Y$ along $\Gamma$. We repeat this until there is no such~$\Gamma$ in any singular leaf.
This does not yet ensure condition~(2) because there may be a point in $L^+\cap L^-$
or a terminal point of $L^+$ or $L^-$ that does not belong to $L^*$.
If some point in $L^+\cap L^-$ does not belong to $L^*$, we split $Y$ at this point. The remaining
issue with terminal points of $L^+$ and $L^-$ is resolved by running the Rips machine in
a certain way described below.

Let $Z$ be the solid part of $Y$, and let $Z=Z_0\supset Z_1\supset Z_2\supset\ldots$ be a sequence in which
every~$Z_k$, $k\geqslant1$, is obtained from $Z_{k-1}$ by a collapse from a free arc.
One can see that no $Z_k$ contains a degenerate band and that the number $s_k$ of singularities
of $\mathcal F_{dx}$ in $Z_k$ may only decrease when $k$ grows. Therefore,
this number stabilizes for sufficiently large $k$. We choose the sequence $(Z_k)$ so that
$\lim_{k\rightarrow\infty}s_k$ is as small as possible.  

There is a unique Rips sequence $Y=Y_0\supset Y_1\supset Y_2\supset\ldots$ such that
$Z_k$ is the solid part of $Y_k$ for all $k=0,1,2,\ldots$ Namely, if $Z_{k-1}\mapsto Z_k$
is a collapse from a free arc $J$ then $Y_k$ is obtained from $Y_{k-1}$ by collapses from all free arcs
of $Y_{k-1}$ contained in $J$ (the endpoints of degenerate bands of $Y_{k-1}$ may subdivide $J$ into several free arcs
of $Y_{k-1}$).

Let $L$ be a singular leaf of $Y$. Denote the intersection $L\cap Y_k$ by $L_k$
and $(L_k)^+$ by $L_k^+$ etc.
Let $\Gamma$ be a connected component of $L^\epsilon$, where $\epsilon\in\{+,-\}$,
and let $L'$ be the connected component of $\Gamma\cup L^*$ that contains $\Gamma$.
Denote also: $L_k'=(L_k^\epsilon\cup L_k^*)\cap L'$,
$\Gamma_k=L_k^\epsilon\cap L'$.

Due to the annulus-free assumption, $\Gamma$ and all components of $L^*$ are simply connected. By
construction, the connected components of $L^*$ that are contained in $L'$ are unbounded.
At least one such component must be present in $L'$ since otherwise regular leaves
in a small neighborhood of $\Gamma$ will be compact.

In transition from $L_{k-1}$ to $L_k$, only the following
changes may occur:
\begin{enumerate}
\item
some terminal edges are removed;
\item
a terminal edge of $L_{k-1}^*$ that shares an endpoint with
$L_{k-1}^+\setminus L_{k-1}^-$ (respectively, $L_{k-1}^-\setminus L_{k-1}^+$) becomes an edge of $L_k^+$
(respectively, $L_k^-$);
\item
some terminal edges of $L_{k-1}^*$ that share an endpoint with
$L_{k-1}^0\setminus(L_{k-1}^+\cup L_{k-1}^-)$ become edges of~$L_k^0$.
\end{enumerate}
This means, in particular, that $L_k'$ will contain at most one connected component of $L_k^\epsilon$,
which we denote by $\Gamma_k$. If $\Gamma$ has more than one intersection point with $L^*$,
then $\core(L_k')$ is non-empty and stays unchanged when $k$ increases, and the edges from $\overline{\Gamma_k\setminus\Gamma_{k-1}}$
belong to $\core(L_k')=\core(L')$. Any inessential branch of $L'$ will eventually be removed by the Rips machine.
$\Gamma\setminus\core(L_k')$ consists of finitely many of them,
so, we will have $\Gamma_k\subset\core(L')$ for all sufficiently large $k$.

Now suppose that $\Gamma\cap L^*$ is a single point $p$. 
If $L'$ has more than one topological end then $\Gamma$ is contained
in an inessential branch of $L'$ and $\Gamma_k\cap\core(L')$ will remain empty.
So, $\Gamma_k$ will be empty for sufficiently large $k$.

Suppose that $L'$ has exactly one topological end. Let $r$ be the unique ray in $L^*$
starting at $p$, and let $e_0,e_1,e_2,\ldots$ be the edges of $r$
numbered in the order they follow in $r$.

Let $k_0\in\mathbb Z$ be a number of steps after which the number of
singularities in $Z_k$ stabilizes, i.e.\ $s_{k_0}=\lim_{k\rightarrow\infty}s_k$.
We claim that $\Gamma_k$ will be empty for $k\geqslant k_0$.

Indeed, suppose otherwise. Change the sequence  $Z_0\supset Z_1\supset\ldots$
starting from $Z_{k_0+1}$ so that each collapse $Z_{k-1}\mapsto Z_{k}$ for $k>k_0$
removes an edge from $\Gamma_{k-1}$. If $\Gamma_k$ will eventually become empty,
the number of singularities in $Z_k$ will decrease, which contradicts to the original choice of the
sequence $(Z_k)$. Thus, $\Gamma_k$ is non-empty for all $k$, and for sufficiently large $k$,
$\Gamma_k$ is obtained from $\Gamma_{k-1}$ by removing an edge
and adding another edge from $r$.

If $j>0$ is the number of edges in $\Gamma_k$ that remains unchanged infinitely long, then, for $k$ large enough, $\Gamma_k$
will consist of the edges $e_{k-q},e_{k-q+1},\ldots,e_{k-q+j-1}$ with some $q$.
This means that $r$ forms an infinite ``regular spiral'' like the one shown in Fig.~\ref{infinitespiral},
which is impossible.
\begin{figure}[ht]
\centerline{\includegraphics{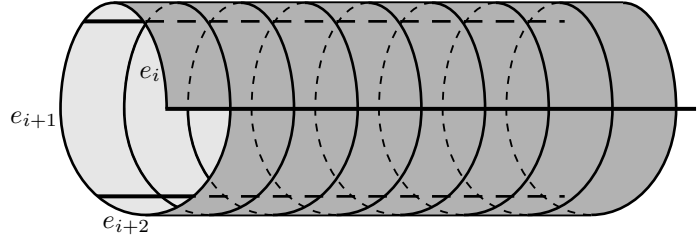}\put(-216,57){$e_i$}\put(-265,40){$e_{i+1}$}\put(-230,0){$e_{i+2}$}}
\caption{An infinite spiral cannot occur in a union of bands}\label{infinitespiral}
\end{figure}

Thus, for any singular leaf $L$ and any connected component $\Gamma$ of $L^\epsilon$ with $\epsilon\in\{+,-\}$,
for sufficiently large $k$, we will have $\Gamma_k\subset\core(L')$,
where the notation is as before. This finally ensures condition~(2) of Definition~\ref{reduced.u.o.b}
for $Y_k$.
\end{proof}

\begin{definition}
Let $Y$ be a reduced union of bands. A vertical edge $e$ of $Y$ is said to be \emph{free}
if $\overline{L\setminus e}$ is disconnected, where $L$ is the singular leaf containing $e$.
\end{definition}

\begin{lemma}
Let $Y=Y_0\supset Y_1\supset Y_2\supset\ldots$ is a sequence of reduced unions
of bands produced by the Rips machine. Then $Y_k$ has a free edge for any $k>0$.
\end{lemma}

\begin{proof}
We can say even more, namely, that every vertical edge of $Y_k$ that
is not a vertical edge of $Y$ is a free edge. Indeed,
let $e$ be such an edge, $L$ the singular leaf of $Y$ that contains $e$.
Then $e\subset L^*$. Since $Y$ is reduced the connected component
of $L^*$ containing $e$ is unbounded and has a single intersection point
with $\sing(L)$. Therefore, $L\setminus e$ has at least two
connected components. Moreover, one of the connected components
contains $\sing(L)$ and another is unbounded, hence, both have
non-empty intersection with $Y_k$ for any $k$.

It remains to note that every $Y_k$ has more vertical edges than $Y_{k-1}$ does.
\end{proof}

Thus, we can add one more assumption on $Y$ that it has a free edge.

Let $(X,\omega)$ be a $2$-complex with a closed $1$-form obtained from $(Y,dx)$
by collapsing to a point every degenerate band and every vertical edge except one free edge,
whose image in $X$ will be denoted by $e$. Denote by $\theta$ the
corresponding map $Y\rightarrow X$. It is trivial to see that $Y$ is a model for $X$.

The oriented $\mathbb R$-tree associated with $X$ will be denoted by $(T,d)$,
the projection $\widetilde X\rightarrow T$ by $\pi$. The model projection $\theta$
can be lifted to a map $\widetilde\theta:\widetilde Y\rightarrow\widetilde X$ such that
$\theta\circ\proj_Y={\proj_X}\circ\widetilde\theta$. Let the orientation of~$T$ be presented
by an orientation map $\ori:T\rightarrow\mathbb R$.
The group $G(X)\cong G(Y)$ is denoted simply by $G$.

Now let $\Phi:(T,d,[\ori])\rightarrow(T,\lambda d,[\lambda\ori])$, $\lambda>1$ be an orientation preserving
isometry of oriented $G$-trees, $\varphi$ the corresponding automorphism of $G$.
We would like to represent them geometrically by a map $\widetilde X\rightarrow\widetilde X$,
but this may not be possible in general unless we add one more technical assumption
stated below.

\begin{lemma}
For some $n>0$ the map $\Phi^n:T\rightarrow T$ preserve every branching
direction of $T$.
\end{lemma}

\begin{proof}
Branch points of $T$ are images of singular fibers of $\widetilde F_\omega$
under $\pi$, so, there are finitely many of them up to the action of $G$.
At every branch point there are only finitely many branches. Thus,
there are only finitely many branching directions in $T$, and $\Phi$
induces a permutation of them, some power of which will be trivial.
\end{proof}

Clearly, $\Phi^n$ is an orientation preserving isometry of $G$-trees
$(T,d,[\ori])\rightarrow(T,\lambda^nd,[\lambda^n\ori])$. Thus, if the
$G$-tree $T$ is self-similar, we may assume that the
self-similarity is realized by a homothety that preserves branching directions.

\subsection{Self-similarity of the first return correspondence}
Recall briefly the additional assumptions introduced above,
which are supposed to hold in the sequel:
\begin{enumerate}
\item
$Y$ is an annulus-free reduced union of bands of thin type;
\item
$Y$ has a free edge;
\item
$X$ is obtained from $Y$ by collapsing to a point every degenerate band and every edge except
one selected free edge $e$;
\item
there is a homothety $\Phi$ of the associated $G$-tree that fixes
branching directions.
\end{enumerate}

\begin{definition}
We say that a continuous map $\phigeom:X\rightarrow X$ is \emph{a geometric realization of
the homothety $\Phi$} if we have $\phigeom^*\omega=\lambda^{-1}\omega$, where
$\lambda$ is the shrink factor of $\Phi$, and $\phigeom$ can be lifted to
a map $\widetilde\phigeom:\widetilde X\rightarrow\widetilde X$ such that
$\pi\circ\widetilde\phigeom=\Phi\circ\pi$. The notation is summarized in
commutative diagram~\eqref{cd}.
\begin{equation}\label{cd}
\raisebox{5.5em}{%
\xymatrix{&G\times T\ar@{->}[r]^{\varphi\times\id}\ar@{->}[d]^{\text{action}}&G\times T\ar@{->}[d]^{\text{action}}\\
&T\ar@{->}[r]^\Phi&T\\
\widetilde Y\ar@{->}[d]^{\proj_Y}\ar@{->}[r]^{\widetilde\theta}
&\widetilde X\ar@{->}[u]_\pi\ar@{->}[r]^{\widetilde\phigeom}\ar@{->}[d]^{\proj_X}
&\widetilde X\ar@{->}[u]_\pi\ar@{->}[d]^{\proj_X}\\
Y\ar@{->}[r]^\theta&X\ar@{->}[r]^\phigeom&X}}\end{equation}
\end{definition}

\begin{proposition}\label{prop-geom-realization-exists}
Let $p$ be an interior point of $e$. Then there exists a geometric realization $\phigeom$ of $\Phi$
such that $\phigeom(p)=p$
and $\phigeom$ is injective on $\phigeom^{-1}(U)$, where $U$ is an open neighborhood of $p$.
\end{proposition}

\begin{proof}
For any leaf $L$ of $\mathcal F_\omega$ the set $\proj_X(\pi^{-1}(\Phi(\pi(\proj_X^{-1}(L)))))$
is also a leaf of $\mathcal F_\omega$, which must become $\phigeom(L)$. For the moment
we denote this leaf by $\Psi(L)$.

By construction, all singularities of $\mathcal F_\omega$ except $e$ are isolated points.
Due to the assumption that $Y$ is reduced every singular leaf of $\mathcal F_\omega$ contains
at most one singularity.

\begin{definition}
A singularity $a$ of $\mathcal F_\omega$ is called \emph{essential} if
$\pi(\proj_X^{-1}(a))$ consists of branch points of $T$.
\end{definition}

We assumed that $\Phi$ preserves branching directions, which implies, in particular,
that it preserves all $G$-orbits of branch points. Therefore, if a leaf $L$
of $\mathcal F_\omega$ contains an essential singularity, then $\Psi(L)=L$.

So, we start by defining $\phigeom$ on all essential singularities of $\mathcal F_\omega$ as identity
and on any other singularity $a$ arbitrarily so as to
have $\phigeom(a)\in\Psi(L)$, where $L$ is the leaf containing $a$.
The map $\phigeom$ defined so far only on singularities is uniquely lifted to a map $\widetilde\phigeom$
defined on the singular set of $\widetilde{\mathcal F}_\omega$ so that
the right column of~\eqref{cd} restricted to this singular set is commutative.

Note that $e$ is an essential singularity, so, $\phigeom$ is identical on $e$.

Now we claim that there is no obstruction to extend $\phigeom$ continuously to the whole
$1$-skeleton of $X$, which consists of the edge $e$ and the image of
the support multi-interval $D$ under the model projection~$\theta$, so as
to keep the right column of~\eqref{cd} restricted to the $1$-skeleton commutative.

Let $\sigma$ be a $1$-cell of $X$ distinct from $e$. Denote by $\widetilde\sigma$
a preimage of $\sigma$ in $\widetilde X$, the endpoints of $\sigma$ by $P$ and $Q$,
and their images under $\widetilde\phigeom$ by $P'$ and $Q'$, respectively
($P$ and $Q$ must be singular points, so the image under
$\widetilde\phigeom$ is already defined for them). Denote also $\Phi(\pi(\widetilde\sigma))$
by $\eta$. We want to define $\widetilde\phigeom$
on $\widetilde\sigma$ and then extend to $\proj_X^{-1}(\sigma)$ equivariantly.

To this end, it suffices to find a transversal arc $\widetilde\sigma'$ in $\widetilde X$
with endpoints $P'$, $Q'$ (we will automatically have $\pi(\widetilde\sigma')=\eta$).

\begin{lemma}\label{xi-q}
For any point $q\in\eta$ there exists a transversal arc $\xi_q$ in $\widetilde X$ such
that $\pi(\xi_q)\subset\eta$ and $\pi(\xi_q)$ covers a small open neighborhood of $q$ in $\eta$.
Moreover, if $q=\pi(P')$ (respectively, $q=\pi(Q')$), then $P'\in\xi_q$ (respectively, $Q'\in\xi_q$),
and $\pi(\xi_q)$ is disjoint from $\partial\eta$ if $q\notin\partial\eta$.
\end{lemma}

\begin{proof}
We use the following general fact. Let $\beta_1,\beta_2$ be two monotonic arcs coming from
a point $q\in T$ in the same direction, i.e. such that $\beta_1\setminus\{q\}$ and $\beta_2\setminus\{q\}$
are contained in the same branch of $T$ at $q$. Then the intersection $\beta_1\cap\beta_2$
is a nontrivial arc. Indeed, if there are $q_1,q_2$ such that
$q_1\in\beta_1\setminus\beta_2$ and $q_2\in\beta_2\setminus\beta_1$ (if $\beta_1\subset\beta_2$
or $\beta_2\subset\beta_1$ the claim is obvious),
then the geodesic arc connecting $q_1$ and $q_2$ must contain a single point $q'$ of $\beta_1\cap\beta_2$,
which must be distinct from $q$ as $q_1$ and $q_2$ lie in the same branch.
The arc between $q$ and $q'$ will be shared by $\beta_1$ and $\beta_2$.

Now let $q\in\eta$.
If $q$ is not a branch point of $T$, then $\pi^{-1}(q)$ is a regular leaf of $\widetilde{\mathcal F}_\omega$.
For an arbitrary open transversal arc $\xi$ intersecting this leaf, the image $\pi(\xi)$ will share
a nontrivial subarc $\eta'$ with $\eta$ of which $q$ is an internal point. We can take $\xi_q=\xi\cap\pi^{-1}(\eta')$.

Let $q\in\eta$ be a branch point of $T$. Denote by $R$ the unique point in $\widetilde\sigma$ such
that $\Phi(\pi(R))=q$. Since $\Phi$ preserves $G$-orbits of branch points there is an element $g\in G$
such that $\pi(g\cdot R)=q$. Denote $g\cdot R$ by $R'$.
In particular, if $R=P\text{ or }Q$, then $R'=P'\text{ or }Q'$, respectively.
Since $\Phi$ preserves branching directions, the image $\pi(\xi)$ of the arc $\xi=g\cdot\widetilde\sigma$ will
proceed from $q$ in the same direction(s) as $\eta$ does. So, we again
can take $\xi_q=\xi\cap\pi^{-1}(\eta\cap\pi(\xi))$.
\end{proof}

It follows from Lemma~\ref{xi-q} that one can find a finite collection of closed transversal
arcs $\xi_1,\xi_2,\ldots,\xi_r$ in~$\widetilde X$ such that:
\begin{enumerate}
\item their images $\pi(\xi_i)$ have disjoint interiors;
\item $\cup_{i=1}^r\pi(\xi_i)=\eta$;
\item $\pi(\xi_i)$ and $\pi(\xi_{i+1})$ share an endpoint that is not
a branch point of $T$ for $i=1,2,\ldots,r-1$;
\item $P'$ is an endpoint of $\xi_1$, and $Q'$ is an endpoint of $\xi_r$.
\end{enumerate}
Therefore, there exists a ``stair-step'' arc $\zeta=\xi_1\cup\alpha_1\cup\xi_2\cup\alpha_2\cup\ldots\cup\xi_r$
connecting $P'$ and $Q'$ such that each $\alpha_i$, $i=1,\ldots,r-1$, is a subset of a regular leaf.
The arc $\zeta$ can be disturbed in a small neighborhood of each $\alpha_i$ to become transversal
to $\widetilde F_\omega$.

Thus, we have shown how to define $\phigeom$ on the $1$-skeleton of $X$, and it remains only to extend it
to the whole $X$. The interior of each $2$-cell of $X$ is foliated by open arcs. Let $\alpha$ be
the closure of such an arc. $\phigeom$ is already defined on $\partial\alpha$, and both points from
$\phigeom(\partial\alpha)$ lie in the same leaf $L$ of $\mathcal F_\omega$. By construction,
$L$ is simply connected, so, we define $\phigeom$ on $\alpha$ so that
$\phigeom(\alpha)$ realizes the shortest path in the leaf between the endpoints $\phigeom(\partial\alpha)$.
This concludes the construction of the map $\phigeom$.

The restriction of $\phigeom$ to the edge $e$ is the identity map. Take any internal point of $e$ for $p$.
We claim, that $\phigeom(p')=p$ implies $p'=p$. Indeed, let $L$ be the leaf of $\mathcal F_\omega$
containing $p$, and let $p'$ be a point from $L\setminus e$. Then there is an open transversal arc $\sigma$
passing through $p'$, which must be taken to an open transversal arc by $\phigeom$. But there is
no open transversal arc passing through $p$, so $\phigeom(p')\ne p$.

The manner in which we defined $\phigeom$ on interiors of $2$-cells ensures that
all points from some open neighborhood of $p$ will also have a single preimage under $\phigeom$.\end{proof}

Now we come back to the union of bands $Y$, which is a model for $X$.
We do not distinguish between the free vertical edge $e\subset X$ and its
preimage in $Y$, using the same notation for both.

For a horizontal arc $\sigma\subset Y$ one of whose endpoints is $p\in e$ and
a real $\mu\in(0,1)$ denote by $\mu\sigma$ the horizontal subarc of $\sigma$
such that $p\in\mu\sigma$ and $|\mu\sigma|=\mu|\sigma|$.

\begin{proposition}\label{self-similarity-of-first-return}
There exists a horizontal arc $\sigma\subset Y$ such that $p\in\sigma$ and
the first return correspondences induced by $\omega$ on $\sigma$
and $\lambda^{-1}\sigma$ are similar.
\end{proposition}

\begin{proof}
It follows from Proposition~\ref{prop-geom-realization-exists} that if $\sigma$ is short enough,
then $\phigeom$ can be chosen so that $\phigeom\eta\subset\eta$, where $\eta=\theta(\sigma)$ and
$\phigeom$ is injective in a small neighborhood of $\eta$.
Therefore, for any leaf $L$, $\phigeom$ takes connected components of $L\setminus\eta$
to connected components of $\phigeom(L)\setminus\phigeom(\eta)$.
It follows that $\mathcal F_\omega$ induces similar first return
correspondences on $\eta$ and $\phigeom(\eta)$ (see Subsection~\ref{first-return}
for definitions), hence $\mathcal F_{dx}$ induces similar
first return correspondences on $\sigma$ and $\theta^{-1}(\phigeom(\eta))$
It remains to notice that,
since $\lambda$ is the shrink factor of $\Phi$, we have $|\phigeom(\eta)|=\lambda^{-1}|\eta|$,
so, $\theta^{-1}(\phigeom(\eta))=\lambda^{-1}\sigma$.
\end{proof}

\subsection{Periodicity of the Rips machine}
We keep using notation and assumptions from the previous subsection.
In particular, $\sigma$ will denote a horizontal arc from Proposition~\ref{self-similarity-of-first-return}.

From now on we don't need the $2$-complex $X$ and the associated $\mathbb R$-tree anymore.
We will only use the self-similarity of the first return correspondence on $\sigma$
and the assumptions that $Y$ is a reduced union of bands.

\begin{lemma}
For any $\mu\in(0,1)$ all blocks of $Y_{\mu\sigma}$ are product ones.
\end{lemma}

\begin{proof}
Since $e$ remains a vertical edge of $Y_k$ for any Rips sequence $(Y_i)$ the
arc $\mu\sigma$ is essential. It remains to apply Proposition~\ref{essential->product}.
\end{proof}

Whatever the sequence $(Y_i)$ is, at some moment a portion of $\sigma$ will be removed,
and the remaining part will contain a horizontal arc starting at $p$
that is \emph{maximal} in the sense that the other its endpoint lies at the opposite
vertical edge of the same band. On this, smaller portion of $\sigma$ the
first return correspondence is still self-similar. So, we may assume from the beginning
that $\sigma$ is maximal.

So far we had a large freedom in choosing the Rips sequence $(Y_i)$. Now we associate
a concrete sequence with the horizontal arc $\sigma$ on which we established self-similarity
of the first return correspondence.
For $t\geqslant0$ let $Y(t)$ be the band complex obtained from $Y$ by doing (recursively) all possible collapses
from a free arc that leave the interior of $e^{-t}\sigma$ untouched. When $t$ grows
$Y(t)$ changes countably many times. Let $t_1=0$ and let $0<t_2<t_3<\ldots$ be all the moments $t$
at which the changes in $Y(t)$ occur.

Denote: $Y_k=Y(t_k)$, $\sigma_k=e^{-t_k}\sigma$. By construction $\sigma_k$ is a maximal
horizontal arc of $Y_k$ for all $k$.

\begin{lemma}\label{t-periodicity}
There is an integer $a>0$ such that we have $t_{k+a}=t_k+\log\lambda$ for all $k>1$.
\end{lemma}

\begin{proof}
By construction $Y(t)$ changes when $t$ passes $t_k$ with $k>1$. This happens
if and only if, for small enough $\varepsilon>0$, the $e^{-t_k+\varepsilon}$-decomposition of $Y(t_k-\varepsilon)$
contains a block consisting from a single long band having a single attaching arc,
and this arc is $\eta=\overline{(e^{-t_k+\varepsilon}\sigma)\setminus(e^{-t_k}\sigma)}$.
This is equivalent to saying that the $e^{-t_k+\varepsilon}$-decomposition of $Y$
contains a block having $\eta$ as the only attaching arc.

The latter fact is detected solely by the first return correspondence induced on $\sigma$,
which is self-similar as stated in Proposition~\ref{self-similarity-of-first-return}.
Therefore, we have $t\in\{t_k\}_{k>1}$ if and only if $t>0$ and $(t+\log\lambda)\in\{t_k\}_{k>1}$.
The claim now follows.
\end{proof}

Following general principles, we may assume that the assertion of Lemma~\ref{t-periodicity}
extends to $k=1$, too.

\begin{lemma}\label{numofblocks}
The number of blocks in $\sigma_k$-decomposition of $Y_k$ does not depend on $k$.
\end{lemma}

\begin{proof}
First, note that the $\sigma_k$-decompositions of $Y$ and of $Y_k$ have the same number of blocks.
Moreover, there is a natural correspondence between the blocks of the two decompositions such that
the blocks of $(Y_k)_{\sigma_k}$ are obtained from respective blocks of $Y_{\sigma_k}$
be removing some number of free arms and have the same attaching arcs.
So, we can replace $Y_k$ by $Y$ in the statement.

Let $L$ be a singular leaf. Since $Y$ is reduced, $\sing(L)$ is connected.
Let $U$ be a small open neighborhood of $\sing(L)$ that retracts to $\sing(L)$.
Denote by $v(L)$ the number of connected components of $U\setminus L$.
(One can show that $v(L)$ is exactly the number of branches at any
point from the $G$-orbit corresponding to $L$ in the associated $G$-tree.)

It is easy to see that the number of blocks of $Y_{\sigma_k}$ to which $\sing(L)$ is adjacent
equals $v(L)+d$, where~$d$ is the number of points in $\sigma_k\cap L$ (thus, $d\in\{0,1,2\}$).
Each block of $Y_{\sigma_k}$ is adjacent to exactly two singularities of $\mathcal F_{dx}$.
Therefore, the number of blocks is equal to
$$1+\frac12\sum_{\sing(L)\ne\varnothing}v(L).$$
\end{proof}

Now we introduce our final assumption on $Y$. We assume that $Y$ is \emph{efficient}
in the following sense: if $Y'$ is another union of bands with a maximal horizontal
arc $\sigma'$ such that the respective first return correspondences on $\sigma$ and $\sigma'$
are similar, then $Y'_{\sigma'}$ has at least as many blocks as $Y_{\sigma}$ does.
Proposition~\ref{obviousprop} implies that replacing $Y$ by an efficient union of bands
inducing the same first return correspondence on~$\sigma$ leaves $\mathcal C(Y)\cap\sigma$ unchanged.

The structure of the whole set $\mathcal C(Y)$ is not necessarily unchanged,
but we are interested only in its Hausdorff dimension, which is higher by one
than that of $\mathcal C(Y)\cap\sigma$. So, it is indeed safe to switch from~$Y$
to an efficient union of bands inducing the same first return correspondence.

\begin{remark}
One can show that efficiency simply means that all singularities of $\mathcal F_{dx}$ are essential.
\end{remark}

\begin{lemma}\label{finitelymany}
There are only finitely  many, up to isomorphism, non-degenerate unions of bands $Z$ such that
\begin{enumerate}
\item
$Z$ can be the solid part of an efficient union of bands inducing, on a maximal
horizontal arc $\xi$, a first return correspondence
similar to the one that $Y$ induces on $\sigma$;
\item
each block of $Z_\xi$ is either a single long band or has no free arms.
\end{enumerate}
\end{lemma}

\begin{proof}
Let $\sim$ be the first return correspondence induced by $Y$ on $\sigma$
and $C$ the set of equivalence classes with respect to $\sim$,
each of which is a subset of $\sigma\times\{+,-\}$.
There are finitely many pairwise disjoint maximal families of the form
$\{\{(x_1+t,\epsilon_1),\ldots,(x_j+t,\epsilon_j)\}\}_{t\in I}$ in $C$, where $I$ is a non-trivial interval (open,
closed, or semi-open). To get the block decomposition of a
union of bands inducing $\sim$ we need at least one block for each such family.
The union of bands will be efficient if each family is presented by exactly one
block, for which there is no obstruction.

If we additionally require that each block has no free arms unless it is a single long band,
then there will be only finitely many choices for each block. Namely,
if $j$ is the number of attaching arcs of the block and $j\geqslant2$, then the number of choices
is equal to the number of finite graphs (viewed up to isomorphism) having exactly $j$ vertices of valence one
and no $2$-valent vertices. If $j=1$ then the block must be a single long band.

When the blocks are chosen there are only finitely many ways
to assemble the solid part $Z$ of a union of bands $Y'$ from them so that
$Y'$ induces the desired first return correspondence.
\end{proof}

For a union of bands $Z$ we denote by $\lambda Z$ the union of bands
obtained from $Z$ by rescaling: $(x,y)\mapsto(\lambda x,y)$.

\begin{proposition}
Under assumptions made above there are integers $a,b>0$ such that
the solid parts of~$Y_k$ and $\lambda^bY_{k+a}$ are isomorphic.
\end{proposition}

\begin{proof}
This follows from Lemmas~\ref{t-periodicity}, \ref{numofblocks}, \ref{finitelymany} and
the fact that, for any $k$, every block of the $\sigma_k$-decomposition of $Y_k$ is either a single
long band or has no free arm.
\end{proof}

\subsection{Finalizing the proof of Theorem~\ref{matintheorem1}}\label{section-example}
Note that removing degenerate bands from $Y$ changes $\mathcal C(Y)$ by
a $1$-dimensional subset, which may be ignored.

Thus, we are left to prove the assertion of the theorem in the case when the foliated
$2$-complex $X$ is a non-degenerate union of bands $Y$
such that some union of bands $Y'\subset Y$ obtained from $Y$ by a finite
sequence of collapses from a free arc is isomorphic to $\lambda^{-1}Y$
with some $\lambda>1$.
This is precisely the situation to which the argument of \cite{bf95,s1}
can be directly extended.

We may additionally assume that all bands of $Y$ are long ones and that none of
them is contained by whole in $Y'$. Assign some
lengths $\ell_1,\ldots,\ell_m$ to them, thus turning $Y$ into an enhanced
union of bands. Then $Y'$ will also become enhanced.
Let $\ell_1',\ldots,\ell_m'$ be the lengths of the corresponding
long bands in $Y'$, where the correspondence is defined
by the isomorphism between $Y$ and $\lambda Y'$. We will have
$(\ell_1',\ldots,\ell_m')=(\ell_1,\ldots,\ell_m)A$ with some
non-zero matrix $A$ having only non-negative integer entries.

Let $\mu>1$ be the Perron--Frobenius eigenvalue of $A$.

\begin{lemma}
We have $\mu<\lambda$.
\end{lemma}

\begin{proof}
Choose $(\ell_1,\ldots,\ell_m)$ to be the eigenvector of $A$ (for right multiplication)
corresponding to the eigenvalue $\mu$. We will have $\ell_i\geqslant0$,
$\ell_i'=\mu\ell_i$. To each band of $Y$ we define its \emph{area}
as the product of the length and the width (it may be zero for some
but not all bands), and the sum of areas of all band will be
called the area of $Y$. We will have $\mathrm{area}(Y')=\lambda^{-1}\mu\,\mathrm{area}(Y)$.
On the other hand, the area of $Y'$ is strictly smaller than that of $Y$,
which implies $\mu<\lambda$.
\end{proof}

Let $\psi:Y\rightarrow Y'$ be a map that realizes an isomorphism between $Y$ and $\lambda Y'$.
Define inductively $Y_1=Y'$, $Y_k=\psi(Y_{k-1})$, $k=2,3,\ldots$. The sequence
$Y\supset Y_1\supset Y_2\supset\ldots$ is produced by the Rips machine.
For any positive length assignments to the bands of $Y$ the lengths of the long bands of $Y_k$
will grow as~$\mu^k$. The widths of bands of $Y_k$ decrease as $\lambda^{-k}$.
Therefore, the Hausdorff dimension of $\mathcal C(Y)=\cap_kY_k$ is equal to
$$1+\frac{\log\mu}{\log\lambda}\in(1,2).$$
This implies that the union of leaves of $\mathcal F_{dx}$ that intersect $\mathcal C(Y)$
has zero measure. These are precisely those that have more than one topological end.
\section{An example of a union of bands with two-ended typical leaves}

\begin{theorem}
There exist uncountably many unions of bands $Y$ of thin type such that
the union of two-ended leaves has full measure in $Y$.
\end{theorem}

This theorem fill follow from Proposition~\ref{prop-two-end} below.

We use notation $\overrightarrow\ell$, $\overrightarrow{\ell'}$, $\overrightarrow{\ell_k}$, $\overrightarrow w$,
$\overrightarrow{w'}$ and $\overrightarrow{w_k}$ for
$$\begin{pmatrix}\ell_1&\ell_2&\ell_3\end{pmatrix},\
\begin{pmatrix}\ell_1'&\ell_2'&\ell_3'\end{pmatrix},\
\begin{pmatrix}\ell_{k1}&\ell_{k2}&\ell_{k3}\end{pmatrix},\
\begin{pmatrix}w_1\\w_2\\w_3\\w_4\\w_5\end{pmatrix},\
\begin{pmatrix}w_1'\\w_2'\\w_3'\\w_4'\\w_5'\end{pmatrix},\text{ and }
\begin{pmatrix}w_{1k}\\w_{2k}\\w_{3k}\\w_{4k}\\w_{5k}\end{pmatrix},$$
respectively. All the coordinates of these columns and rows will be positive reals.

Let $Z=Z(\overrightarrow w,\overrightarrow\ell)$ be an enhanced union of bands shown in Fig.~\ref{z(w,l)}. It consists of four
bands $B_1$, $B_2$, $B_3$, and $B_4$ having dimensions $w_1\times\ell_1$,
$(w_2+w_3)\times\ell_2$, $(w_1+w_2+w_3+w_4+w_5)\times\ell_3$,
and $(w_1+w_2+w_3+w_4+w_5)\times\ell_1$, respectively.

\begin{remark}
The band $B_4$ is actually not needed. If we identify its bases and remove the band, then
the intersection of $\mathcal C(Z)$ with other bands will not change. The obtained
union of bands will consist of just three bands and, in the terminology of \cite{glp94},
it will be the suspension of the following
system of partial isometries on the interval $D=[0,2w_1+2w_2+2w_3+w_4+w_5]$:
$$\begin{aligned}\relax
[0,w_1]&\leftrightarrow
[w_1+2w_2+2w_3+w_4+w_5,2w_1+2w_2+2w_3+w_4+w_5],\\
[w_1+w_2,w_1+2w_2+w_3]&\leftrightarrow
[w_1+w_2+w_3+w_4,w_1+2w_2+2w_3+w_4],\\
[0,w_1+w_2+w_3+w_4+w_5]&\leftrightarrow
[w_1+w_2+w_3,2w_1+2w_2+2w_3+w_4+w_5].
\end{aligned}$$
\end{remark}

\begin{figure}[ht]
\centerline{\includegraphics{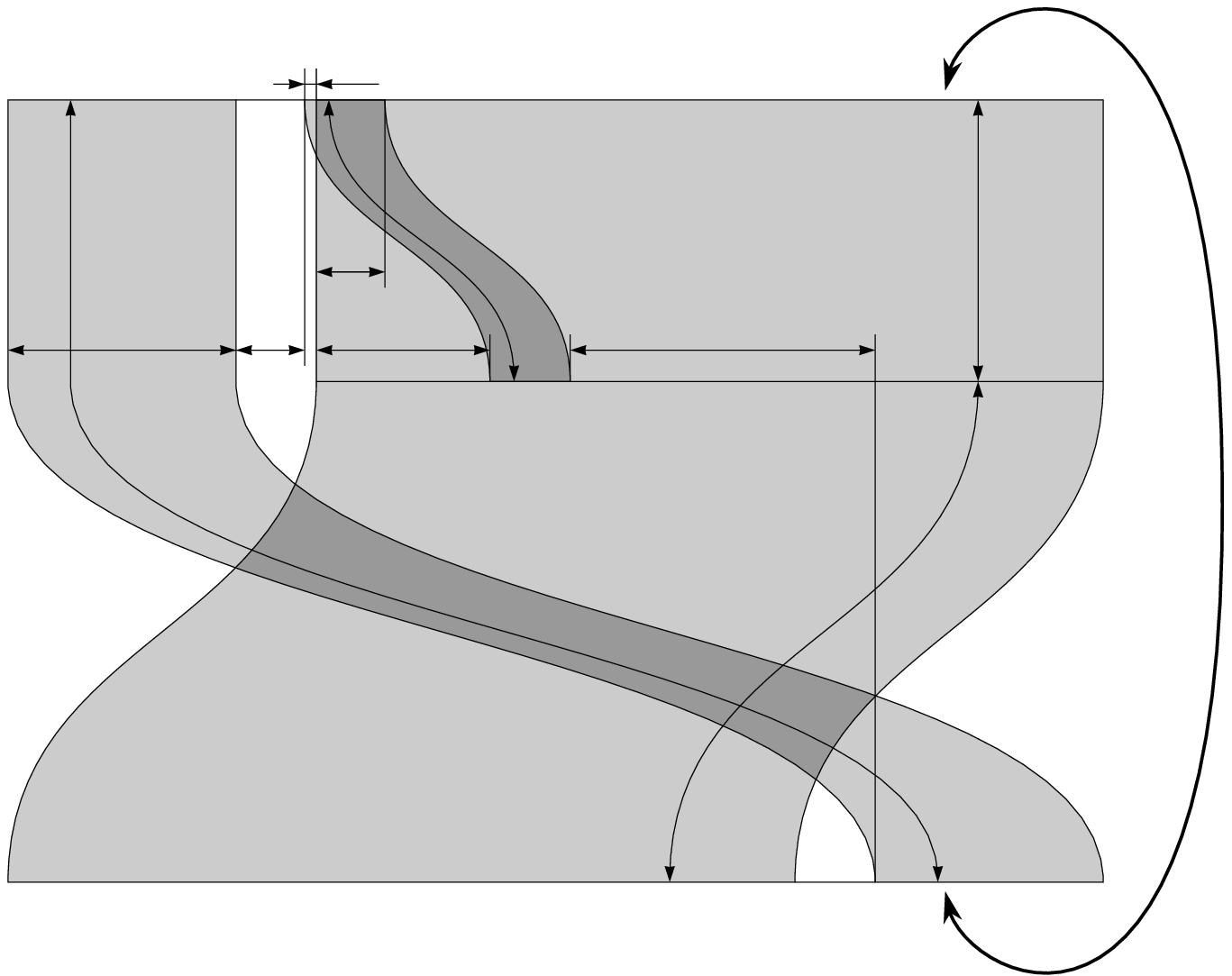}\put(-358,205){$\scriptstyle w_1$}\put(-312,205){$\scriptstyle w_2$}\put(-284,290){$\scriptstyle w_3$}\put(-268,205){$\scriptstyle w_4$}\put(-286,230){$\scriptstyle w_2$}\put(-168,205){$\scriptstyle w_5$}\put(-352,160){$\scriptstyle\ell_1$}\put(-235,215){$\scriptstyle\ell_2$}\put(-100,150){$\scriptstyle\ell_3$}\put(-90,230){$\scriptstyle\ell_1$}\put(-15,140){\begin{sideways}identify\end{sideways}}}
\caption{The union of bands $Z(\protect\overrightarrow w,\protect\overrightarrow\ell)$}\label{z(w,l)}
\end{figure}

Now we define:
$$\begin{aligned}
A(m,n)&=\begin{pmatrix}
m+3&m+3&(m+3)(n+1)-1\\
0&1&1\\
m+2&m+1&(m+2)(n+1)
\end{pmatrix},\\
B(m,n)&=\begin{pmatrix}
n+3&2n+5&2n+6&n+3&n+3\\
1&3&4&2&1\\
1&1&0&0&0\\
n+2&2n+4&2n+5&n+2&n+2\\
m(n+5)&m(2n+9)-1&2m(n+5)-1&m(n+5)&m(n+4)
\end{pmatrix}.\end{aligned}$$

Denote: $\mathbb R_+=(0,\infty)$.

\begin{lemma}\label{lem-Rips-step}
Let $\overrightarrow\ell,\overrightarrow{\ell'}\in(\mathbb R_+)^3$, $\overrightarrow w,\overrightarrow{w'}\in(\mathbb R_+)^5$ be related as follows:
$$\overrightarrow\ell A(m,n)=\overrightarrow{\ell'},\quad
\overrightarrow w=B(m,n)\overrightarrow{w'},$$
where $m,n$ are natural numbers.
Then the enhanced union of bands $Z(\overrightarrow{w'},\overrightarrow{\ell'})$ is isomorphic
to a one obtained from $Z(\overrightarrow w,\overrightarrow\ell)$ by several collapses from
a free arc.
\end{lemma}

\begin{proof}
\begin{figure}[ht]
\centerline{\includegraphics{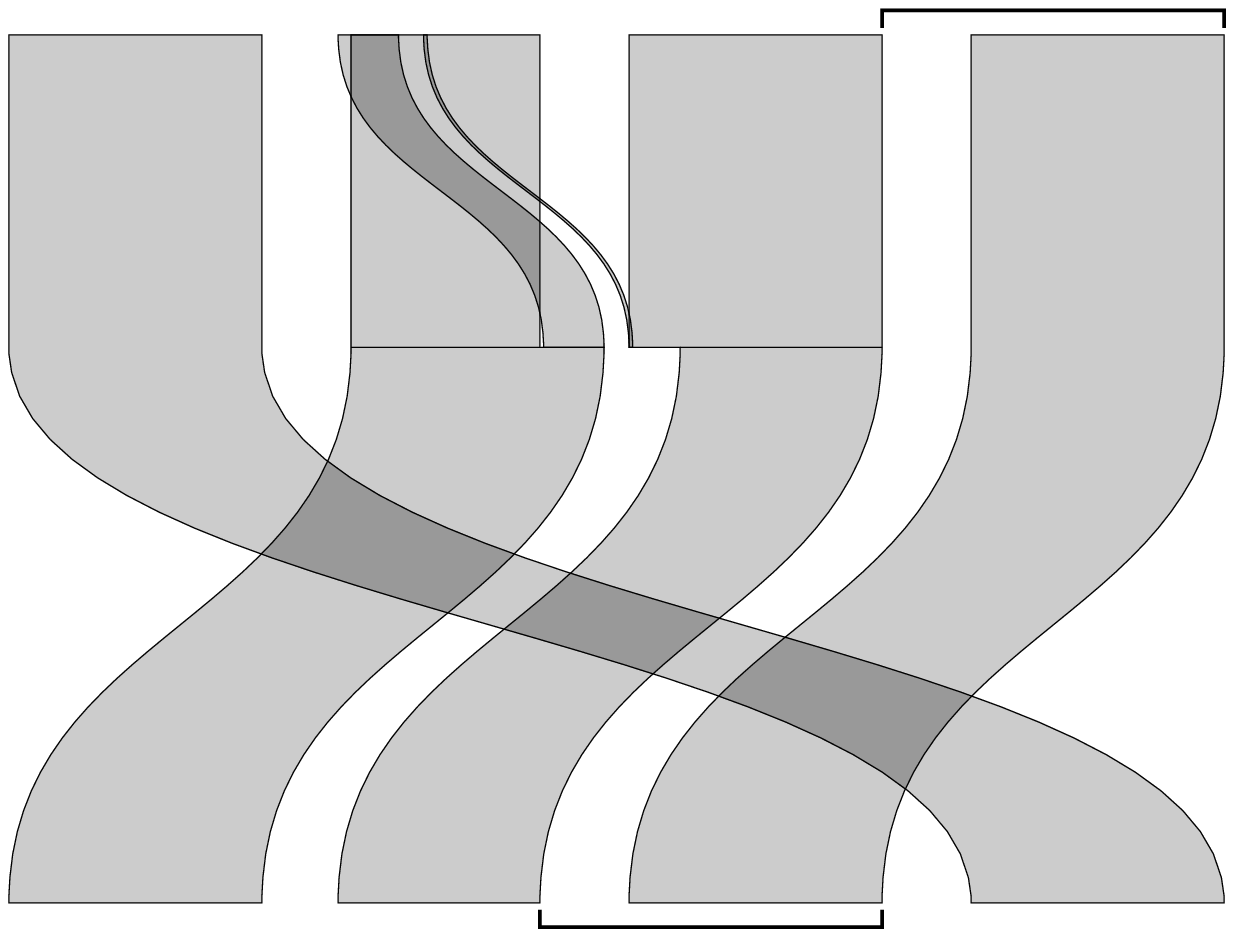}\put(-100,292){$\times m$}\put(-195,15){$\times m$}}
\centerline{\includegraphics[scale=0.5]{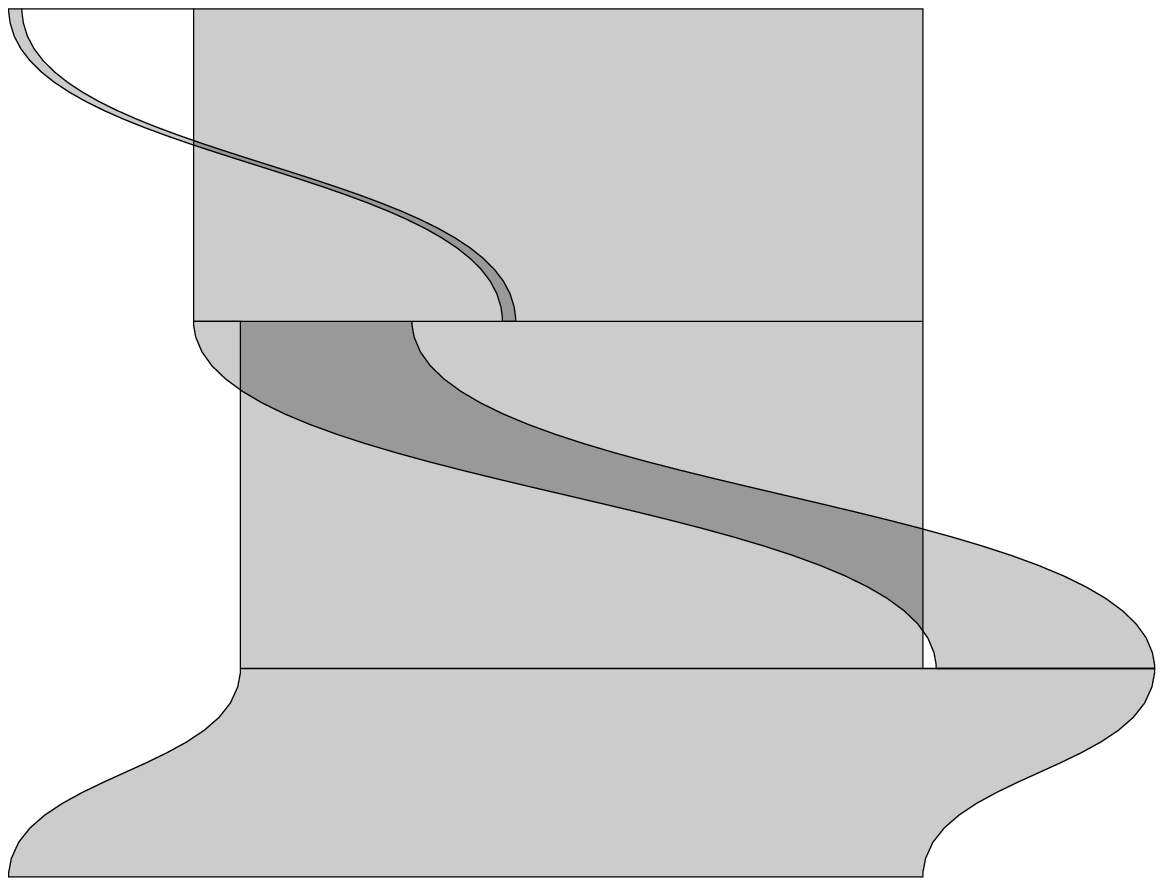}}
\caption{Running the Rips machine on $Z(\protect\overrightarrow w,\protect\overrightarrow\ell)$, first stage}\label{first-stage}
\end{figure}
It is a direct check that the collapses shown in Figs.~\ref{first-stage} and~\ref{second-stage} do the job.
$Z(\overrightarrow{w'},\overrightarrow{\ell'})$ is obtained from
$Z(\overrightarrow w,\overrightarrow\ell)$ in two stages. On the top Fig.~\ref{first-stage},
we show what remains of $Z(\overrightarrow w,\overrightarrow\ell)$ after the first stage,
the bottom picture shows the obtained union of bands in a simplified form.
\begin{figure}[ht]
\centerline{\includegraphics{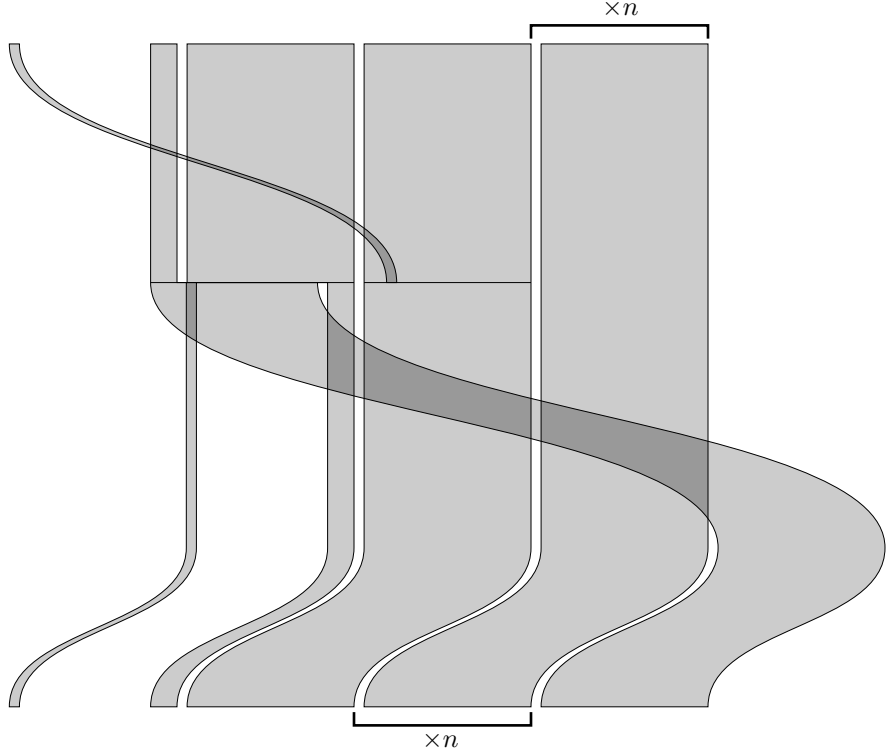}\put(-117,292){$\times n$}\put(-185,15){$\times n$}}
\caption{Running the Rips machine on $Z(\protect\overrightarrow w,\protect\overrightarrow\ell)$, second stage}\label{second-stage}
\end{figure}
Fig.~\ref{second-stage} demonstrates what is removed at the second stage. The obtained
union of bands is isomorphic to $Z(\overrightarrow{w'},\overrightarrow{\ell'})$.

The portions of the pictures marked `${\times}m$' and `${\times}n$' should be repeated $m$ and $n$ times, respectively.
The top and bottom lines should be identified in each picture.\end{proof}

\begin{lemma}\label{w-exist}
Let $m_0,m_1,m_2,\ldots$, $n_0,n_1,n_2,\ldots$ be arbitrary infinite sequences of natural numbers. Then
there exists an infinite sequence $\overrightarrow{w_0},\overrightarrow{w_1},\overrightarrow{w_2},\ldots$
of points from $(\mathbb R_+)^5$ such that
$$\overrightarrow{w_k}=B(m_k,n_k)\overrightarrow{w_{k+1}}.$$
Such a sequence is unique up to scale.
\end{lemma}

\begin{proof}
Denote by $S^{r-1}$ the unit sphere in $\mathbb R^r$.
For an $r\times r$-matrix $B$ with non-negative coefficients, define a map $\eta_B$ from the
set $S^{r-1}\cap\mathbb R_+^r$ to itself
by
$$\eta_B(v)=\frac{Bv}{|Bv|}.$$
(We use the standard Euclidean metric in $\mathbb R^r$.)

One can easily verify that $B(m,n)=B'(m,n)\cdot B''$, where 
$$B'(m,n)=\begin{pmatrix}n+1&n&0&1&0\\0&1&0&0&1\\0&0&1&0&0\\n&n&0&1&0\\m(n+1)&m(n+1)-1&m&m&m\end{pmatrix},\quad
B''=\begin{pmatrix}1&1&1&1&1\\0&1&1&0&0\\1&1&0&0&0\\2&4&5&2&2\\1&2&3&2&1\end{pmatrix}.$$
The mapping $\eta_{B'(m,n)}$ is non-expanding for any $m,n\geqslant1$, the map $\eta_{B''}$ is a contraction.
Therefore, $\eta_{B(m,n)}$ is a contraction with contraction coefficient not smaller than that of $\eta_{B''}$. Hence,
we can set
$$\overrightarrow{w_0}=\lim_{k\rightarrow\infty}\bigl((\eta_{B(m_1,n_1)}\circ\eta_{B(m_2,n_2)}\circ\eta_{B(m_3,n_3)}
\circ\ldots\circ\eta_{B(m_k,n_k)})v_k\bigr),$$
where $v_k\in\mathbb R_+^5$, $k=1,2,\ldots$ are arbitrary, the limit always exists and does not depend
on the choice of $v_k$'s. This also implies the uniqueness
of $\overrightarrow{w_0}$ up to a multiple.\end{proof}

Denote by $Z_{\mathbf m,\mathbf n}$ the union of bands $Z(\overrightarrow{w_0},\overrightarrow{\ell_0})$,
where $\overrightarrow{\ell_0}=(1,1,1)$, and $\overrightarrow{w_0}$ is determined by
$\mathbf m=(m_0,m_1,\ldots)$, $\mathbf n=(n_0,n_1,\ldots)$ as in Lemma~\ref{w-exist}. Define recursively
$$\overrightarrow{\ell_{k+1}}=\overrightarrow{\ell_k}\cdot A(m_k,n_k).$$

We now claim that if $\mathbf m$ and $\mathbf n$ grow fast enough, then a typical leave in $Z_{\mathbf m,\mathbf n}$
will have two ends.

\begin{proposition}\label{prop-two-end}
If for all $k\geqslant0$ we have $m_k=n_k\leqslant m_{k+1}/2$, then
the union of leaves in $Z_{\mathbf m,\mathbf n}$ that are not two-ended trees has
zero measure.
\end{proposition}

\begin{proof}
Denote for short:
$$A_k=A(m_k,n_k),\ B_k=B(m_k,n_k),\ Z_k=Z(\overrightarrow{w_k},\overrightarrow{\ell_k}).$$
It follows from Lemma~\ref{lem-Rips-step} that $Z_{k+1}$ can be identified with an enhanced union of bands $Z_k$
obtained from $Z_k$ by a few collapses from a free arc. Thus, we can regard
$Z_0=Z_{\mathbf m,\mathbf n},Z_1,Z_2,\ldots$ as a  Rips sequence.
Denote by $S_k$ the total area of $Z_k$:
$$S_k=\overrightarrow{\ell_k}\cdot C\cdot \overrightarrow{w_k},$$
where
$$C=\begin{pmatrix}2&1&1&1&1\\0&1&1&0&0\\1&1&1&1&1\end{pmatrix}.$$

We claim that under the assumptions of the Proposition we have
\begin{equation}\label{limit}
\lim_{k\rightarrow\infty}S_k>0.
\end{equation}
Indeed, it can be checked directly that the matrix
$$m_k\left(A_kC-\Bigl(1-\frac2{m_k}\Bigr)CB_k\right)B_{k+1}$$ has only positive entries
for all $k\geqslant0$
since they can be expressed as polynomials in $m_k$ and $(m_{k+1}-2m_k)$ with positive coefficients.
Therefore,
$$S_{k+1}-\Bigl(1-\frac2{m_k}\Bigr)S_k=\overrightarrow{\ell_k}\left(A_kC-\Bigl(1-\frac2{m_k}\Bigr)CB_k\right)B_{k+1}\overrightarrow{w_{k+2}}>0,$$
which can be rewritten as
$$S_{k+1}>\Bigl(1-\frac2{m_k}\Bigr)S_k.$$
Since $m_k$ grows exponentially with $k$, we have
$$\sum_{k=0}^\infty\frac2{m_k}<\infty,$$
which implies~\eqref{limit}.

Thus, the union $\mathcal C(Z_{\mathbf m,\mathbf n})$ of cores of all leaves in $Z_{\mathrm m,\mathrm n}$ has
positive area, hence, so does the union of all two-ended leaves. It follows from Lemma~\ref{w-exist} that
the foliation $\mathcal F_{dx}$ on $Z_{\mathbf m,\mathbf n}$ is ``uniquely ergodic'', which means
that any measurable union of leaves is either of zero measure or of full measure. We have excluded
the first option for the union of two-ended leaves, so it is of full measure.\end{proof}

\end{document}